\documentclass[11pt, a4paper]{article}

\usepackage[english]{babel}
\usepackage[left=2.5cm, right=2.5cm, top=2.8cm, bottom=2.8cm]{geometry}
\usepackage{titlesec}          
\usepackage{fancyhdr}          
\usepackage[titles]{tocloft}   
\usepackage{parskip}           
\usepackage{imakeidx}          
\makeindex[]

\usepackage{mathtools}         
\usepackage{amssymb, amsthm}   
\usepackage{dsfont, bbm}       
\usepackage{cancel}            

\usepackage{booktabs, array, multirow, bigstrut} 
\usepackage{enumitem}          
\setlist[itemize]{noitemsep}   
\usepackage{comment}

\usepackage{graphicx,xcolor}
\usepackage{float}             
\usepackage{tikz, tikz-cd}     
\usetikzlibrary{arrows.meta}   

\usepackage[colorlinks=true, 
linkcolor=teal, citecolor=purple, urlcolor=purple]{hyperref}

\titleformat{\paragraph}[runin]
  {\normalfont\normalsize\bfseries}
  {}
  {0pt}
  {}
\titlespacing*{\paragraph}{0pt}{6pt}{1em}

\numberwithin{equation}{section}
\allowdisplaybreaks

\DeclareMathOperator{\rank}{rk}

\DeclareMathOperator{\Dim}{dim} 

\DeclareMathOperator{\Span}{Span}

\DeclareMathOperator{\CC}{\mathbb{C}}

\newcommand{\size}[1]{\left| #1 \right|}

\newcommand{\vs}{\vspace*{.5em}}

\newcommand\restr[2]{{%
  \left.\kern-\nulldelimiterspace 
  #1 \vphantom{\small|} 
  \right|_{#2} 
  }}

\makeatletter
\def\mathcenterto#1#2{\mathclap{\phantom{#1}\mathclap{#2}}\phantom{#1}}
\let\old@widetilde\widetilde
\def\widetildeto#1#2{\mathcenterto{#2}{\old@widetilde{\mathcenterto{#1}{#2\,}}}}
\let\old@widehat\widehat
\def\widehatto#1#2{\mathcenterto{#2}{\old@widehat{\mathcenterto{#1}{#2\,}}}}

\newcommand*\closure[1]{\overline{#1}}
\makeatother

\theoremstyle{plain}
\newtheorem{theorem}{Theorem}[section]
\newtheorem{proposition}[theorem]{Proposition}
\newtheorem{lemma}[theorem]{Lemma}
\newtheorem{corollary}[theorem]{Corollary}

\theoremstyle{definition}
\newtheorem{definition}[theorem]{Definition}

\newtheorem{example}[theorem]{Example}
\newtheorem{remark}[theorem]{Remark}
\newtheorem{notation}[theorem]{Notation}

\newtheorem{question}{Question}
\newtheorem*{question*}{Question}

\newtheorem{theoremA}{Theorem}

\newtheorem{theoremB}{Theorem}

\newtheorem{theoremC}{Theorem}

\title{
Decomposing Determinantal Varieties from Statistics via \\Matroid Theory
}
\author{Per Alexandersson, Yulia Alexandr, Emiliano Liwski, \\Fatemeh Mohammadi, Pardis Semnani}
\date{}

\begin{document}

\maketitle

\begin{abstract}
We study determinantal varieties from conditional independence models with hidden variables, focusing on their irreducible decompositions, dimensions, degrees, and Gröbner bases. 
Each variety encodes a collection of matroids, whose flats capture algebraic dependencies among variables. 
Using this approach, we provide a systematic description of the components, their dimensions, and defining equations, and introduce a combinatorial framework for computing the degree of the determinantal variety. 
Our approach highlights the central role of matroidal structures in the study of determinantal varieties and extends beyond the reach of current computational techniques.
\end{abstract}



\section{Introduction}\label{sec:intro}

\paragraph{Motivation.}
Determinantal varieties play a central role in algebraic geometry and combinatorial algebra, arising naturally in tensor models, hypergraph and matroid theory, and the study of structured matrices~\cite{bruns2003determinantal}. 
They provide a unifying framework for understanding algebraic dependencies among collections of variables and the combinatorial structures that underlie them. 
A fundamental problem is to describe the \emph{irreducible components} of such varieties and to characterize algebraic invariants such as dimension, degree, and Gr\"obner bases. 
While these questions are well-understood for determinantal ideals defined by minors of generic matrices, the situation becomes considerably more intricate when additional combinatorial constraints are imposed~\cite{clarke2020conditional, herzog2010binomial,  HS04, pfister2019primary, Rauh}.
In these structured cases, especially those induced by hypergraphs or matroids, the corresponding varieties often exhibit decompositions that are challenging to describe with standard algebraic methods.

\paragraph{From conditional independence to determinantal ideals.}
This paper studies a family of determinantal varieties defined by structured minors that emerge in the algebraic formulation of \emph{conditional independence (CI) models with hidden variables}; see, for instance,~\cite{alexandr2025decomposing, clarke2024liftable, clarke2022conditional, clarke2020conditional, pfister2019primary}. 
Conditional independence is a cornerstone of statistical modeling~\cite{maathuis2018handbook, studeny2006probabilistic}, and its algebraic form has been extensively investigated in algebraic statistics~\cite{drton2008lectures, sullivant2023algebraic} and combinatorics~\cite{andersson1993lattice, caines2022lattice}. 
In this framework, probability distributions satisfying a given collection of CI statements correspond to algebraic varieties defined by polynomial constraints, known as \emph{CI ideals}. 
When some variables are unobserved (hidden), the resulting ideals are typically generated by minors of different sizes, and their decompositions encode algebraic analogues of statistical properties such as the intersection property~\cite{pearl1989conditional, peters2015intersection, Steudel-Ay}.

\paragraph{The model.}
We focus on the case where one hidden variable has state space of size $t-1$. 
The corresponding CI model is
\begin{equation}\label{model}
\mathcal{C} = \{\, X \mathrel{\perp\!\!\!\perp} Y_{1} \mid Y_{2},\quad
X \mathrel{\perp\!\!\!\perp} Y_{2} \mid \{ Y_{1}, H \} \,\}.
\end{equation}
This family gives rise to determinantal ideals generated by $2$- and $t$-minors, depending on the role of the hidden variable~$H$. 
The associated varieties admit a combinatorial description in terms of grids and hypergraphs, and their decomposition yields new families of prime mixed determinantal ideals whose Gr\"obner bases consist precisely of their generating minors; see also~\cite{HS04}. 
Our results extend and unify previous decompositions of such varieties within a general matroidal framework.

\subsection{Definitions}\label{sub_intro_def}

\subsubsection{Algebraic formulation}
We translate the conditional independence statements in~\eqref{model} into their algebraic counterparts. 
Each independence condition corresponds to a rank constraint on a submatrix of the joint probability tensor of the involved random variables, leading to a family of mixed determinantal ideals whose geometric properties reflect the structure of the underlying model.

\begin{definition}[CI ideal]\label{setup}
Let $X$, $Y_{1}$, and $Y_{2}$ be observed random variables, and let $H$ be a hidden variable taking values in $\mathcal{H}$, with $|\mathcal{H}| = t - 1$. 
Set $\lvert \mathcal{X} \rvert = d$, $\lvert \mathcal{Y}_{1} \rvert = k$, and $\lvert \mathcal{Y}_{2} \rvert = \ell$. 
Define the $k\times \ell$ matrix
\begin{eqnarray}\label{matrix}
\mathcal{Y} =
\begin{pmatrix}
1 & k+1 & \cdots & (\ell-1)k+1 \\
2 & k+2 & \cdots & (\ell-1)k+2 \\
\vdots & \vdots & \ddots & \vdots \\
k & 2k & \cdots & \ell k
\end{pmatrix}.
\end{eqnarray}
For $i \in [k]$ and $j \in [\ell]$, let
\[
R_{i} = \{\, i,\, k+i,\, \ldots,\, (\ell-1)k+i \,\}, 
\quad 
C_{j} = \{\, (j-1)k+1,\, \ldots,\, jk\,\},
\] 
denote the individual rows and columns of $\mathcal{Y}$ and set
\[
\Delta = \bigcup_{i \in [k]} \binom{R_i}{t} \cup \bigcup_{j \in [\ell]} \binom{C_j}{2}.
\]
For a $d\times k\ell$ matrix $X=(x_{i,j})$ of indeterminates, define
\[
I_{\Delta} = \big\langle [A \mid B]_X : A \subset [d],\, B \in \Delta,\, |A| = |B| \big\rangle \subset \CC[X],
\]
where $[A \mid B]_X$ denotes the minor of $X$ obtained by selecting the rows indexed by $A$ and the columns indexed by $B$, and 
\[
V_{\Delta} = \{\, Y \in \CC^{d\times k\ell} : \operatorname{rank}(Y_F) < |F| \text{ for all } F\in\Delta \,\},
\]
where $Y_F$ denotes the submatrix of $Y$ consisting of the columns indexed by $F$.
The ideal $I_{\Delta}$ corresponds to the model in~\eqref{model}; 
see~\cite{clarke2022conditional, clarke2020conditional}.
\end{definition}

\subsubsection{Matroid definitions}\label{sec:preliminaries}

We briefly recall the basic notions of matroid theory and refer to~\cite{Oxley} for further details.  
\begin{definition}
A \emph{matroid} $M$ consists of a ground set~$[n]$ together with a collection $\mathcal{I}(M)$ of subsets of~$[n]$, called \emph{independent sets}, satisfying:
\begin{enumerate}[label=(\roman*)]
\item $\emptyset \in \mathcal{I}$,
\item if $I \in \mathcal{I}$ and $I' \subseteq I$, then $I' \in \mathcal{I}$,
\item if $I_{1}, I_{2} \in \mathcal{I}$ with $|I_{1}| < |I_{2}|$, then there exists $e \in I_{2} \setminus I_{1}$ such that $I_{1} \cup \{e\} \in \mathcal{I}$.
\end{enumerate}
\end{definition}

Several equivalent descriptions of matroids exist (via circuits, rank, or bases). We recall them below.

\begin{definition}
Let $M$ be a matroid on $[n]$.
\vspace*{-.5em}

\begin{itemize}[label=$\blacktriangleright$]
\item A subset of $[n]$ that is not independent is \emph{dependent}. The set of all dependent subsets is $\mathcal{D}(M)$.\vs
\item A \emph{circuit} is a minimally dependent subset of $[n]$. The set of circuits is $\mathcal{C}(M)$.\vs
\item For $F \subseteq [n]$, its \emph{rank}, denoted $\rank(F)$, is the size of the largest independent subset of $F$. The \emph{rank of $M$} is $\rank([n])$.\vs
\item A \emph{basis} is a maximal independent subset of $[n]$. The set of all bases is $\mathcal{B}(M)$.\vs
\item An element $x \in [n]$ lies in the \emph{closure} of $F \subseteq [n]$, written $x \in \closure{F}$, if $\rank(F \cup \{x\}) = \rank(F)$. A set $F$ is a \emph{flat} if $F = \closure{F}$.\vs
\item An element $x \in [n]$ is a \emph{loop} if $\rank(\{x\}) = 0$.
\end{itemize}
\end{definition}

\subsection{Contributions}

This work focuses on the following question.

\begin{question}\label{question decomposition}
Determine the irreducible decomposition of the variety $V_{\Delta}$.
\end{question}

The irreducible decomposition of $V_{\Delta}$ has been obtained only for special parameter values, for example $t=\ell$ in~\cite{clarke2020conditional}, $t=3$ in~\cite{clarke2022conditional}, 
and $k=2$, $t=d$ in~\cite{alexandr2025decomposing}; see also~\cite{clarke2021matroid,  clarke2024liftable, pfister2019primary}. 
For general $(t,k,\ell,d)$, the problem remains computationally intractable.

\paragraph{A matroid-theoretic approach.}
To overcome these computational difficulties, we adopt a \emph{matroid-theoretic framework}, which provides a combinatorial description of the decomposition of $V_{\Delta}$ and an explicit characterization in the case $k=2$. 
This extends~\cite{alexandr2025decomposing} and highlights the fundamental role of matroidal structures in the geometry of conditional independence varieties.

More generally, CI ideals can be defined with an additional parameter $s \geq 2$; see Section~\ref{sec:quasi_product}. In this work we restrict to the case $s=2$. Matroids also arise in a different setting, when $d\leq s+t-3$, where the associated CI variety is the circuit variety of a matroid; see Theorem~\ref{thm:dleq s+t-3} and \cite{clarke2021matroid}. While this manuscript was being finalized, the third author used similar matroid-theoretic ideas for the special~case~$s= t= n$~in~\cite{liwski2025decomposing}.

\paragraph{Outline of the strategy.}
Our analysis proceeds as follows.

First, we introduce auxiliary varieties $U_S$ for subsets $S\subset[k\ell]$, defined using matroidal rank conditions in Definition~\ref{def: variety US}, which are  subvarieties of $V_{\Delta}$ with ambient space $(\CC^{d})^{k\ell}$. 
The varieties $U_{S}$ are {\em quasi-affine}; that is, they are the intersection of a Zariski-closed subset with a Zariski-open subset. 
We then define $V_S$ as the Zariski closure of $U_S$ in $(\CC^{d})^{k\ell}$ and show that our variety $V_\Delta$ decomposes into a union of such varieties $V_S$. In general, the varieties $V_{S}$ may not be irreducible, so we introduce a further decomposition of $V_S$ into irreducible components $V_{S,i}$. As a result, we obtain an irredundant irreducible decomposition of $V_\Delta$ in terms of the smaller components $V_{S,i}$, where the subsets $S$ appearing in the decomposition are the \emph{admissible subsets} introduced in Definition~\ref{admissible}.

To explicitly describe this decomposition of $V_\Delta$, we need to focus on the varieties $V_S$. To decompose the varieties $V_S$, it suffices to decompose the quasi-affine varieties $U_S$ (see Theorem~\ref{irreducible}). To this end, we further reduce our problem by introducing quasi-affine varieties $F_S$, defined for each admissible subset $S \subset [k\ell]$. These varieties live in the lower-dimensional space $(\CC^d)^{\ell}$ and also have a matroid-theoretic interpretation. In Theorem~\ref{corresp} we establish an explicit one-to-one correspondence between the irreducible components of $F_S$ and $U_S$. This multi-step reduction enables a tractable analysis of $V_\Delta$ via the geometry of the smaller varieties $F_S$, leading to an explicit decomposition in the case when $k=2$ and $\ell \geq t$ are arbitrary. Moreover, this approach provides a general framework for answering Question~\ref{question decomposition}, substantially simplifying the problem and offering new insight into the geometry of the~components~of~$V_{\Delta}$.

\medskip
\noindent
\textbf{\large Our contributions.}  
We now summarize the main contributions of this paper. All results concern the irreducible decomposition, dimension, and degree of the variety $V_{\Delta}$.

\begin{theoremA}\label{theoremA}
The variety $V_{\Delta}$ admits the following irredundant irreducible decomposition:
\begin{equation*}
V_{\Delta} = \bigcup_{S} \bigcup_{i=1}^{a_{S}} V_{S,i}, \tag{Theorem~\ref{irreducible}}
\end{equation*}
where the union runs over all admissible subsets $S \subset [k\ell]$ such that $V_{S} \neq \emptyset$, and where $V_{S,i}$ denote the irreducible components of $V_{S}$. Moreover:
\begin{itemize}
\item[1.] We establish a one-to-one correspondence between the irreducible components of $V_{S}$ and those of~$F_{S}$, made explicit in the proof. \hfill (Theorem~\ref{corresp})\vs
\item[2.] Under this correspondence,
\begin{equation*}
\Dim(V_{S,i}) = \Dim(F_{S,i}) + \ell(k-1) - \lvert S \rvert, \tag{Corollary~\ref{dimensions}}
\end{equation*}
where $F_{S,i}$ maps to $V_{S,i}$ under the stated correspondence.\vs
\item[3.] We explicitly 
show how any finite set of generators of the ideal $I(F_{S,i})$ 
yields a complete set of defining equations of $V_{S,i}$. \hfill (Proposition~\ref{equations})
\end{itemize}
\end{theoremA}

By Theorem~\ref{theoremA}, the irreducible decomposition of $V_{\Delta}$, the dimensions of its components, and their defining equations all reduce to the respective questions for the varieties $F_{S}$. We emphasize that the passage through the varieties $F_{S}$ is not merely a convenient strategy for decomposing $V_{\Delta}$, but rather a necessary and intrinsic step for understanding this decomposition.

We successfully apply our method to obtain the irreducible decomposition of $V_{\Delta}$ in two cases: (i) $k=2$ and arbitrary $\ell \geq t$; (ii) $t=\ell$ and arbitrary $k$.
The decomposition for $t=\ell$ was previously derived in~\cite{clarke2020conditional}; however, our approach provides substantially simpler proofs, offering additional insight into possible generalizations.

\begin{theoremB}\label{theoremB}
For $t=\ell$, $V_{\Delta}$ has the following irredundant irreducible decomposition:
\begin{equation*}
V_{\Delta} \;=\; V_{\emptyset} \,\cup\, \bigcup_{S\in \mathcal{R}} V_{S}, \tag{Corollary~\ref{desc del teo}}
\end{equation*}
where 
\[
\mathcal{R}=\{S\subset [k\ell]:\ \text{$\size{R_{i}\cap S}=1$ for all $i\in [k]$ and $C_{j}\not\subset S$ for all $j\in [\ell]$}\}.
\]
Moreover, the dimensions of its irreducible components are given by
\[
\Dim(V_{\emptyset})=\ell(k+d)-d-1, 
\qquad 
\Dim(V_{S})=\ell(k+d-1)-k 
\ \text{for each } S\in \mathcal{R}. 
\tag{Proposition~\ref{prop:t=l-components}}
\]
\end{theoremB}
For $k = 2$, the only case in which the decomposition of $V_{\Delta}$ was previously known is when $d = t$, obtained in~\cite{alexandr2025decomposing}. We note, however, that the arguments in that proof do not extend to the general case $d \geq t$. Using our framework, we derive the explicit irredundant irreducible decomposition of $V_{\Delta}$ for $k = 2$ and arbitrary $d \geq t$.

\begin{theoremC}\label{theoremC}
For $k = 2$, the irredundant irreducible decomposition of $V_{\Delta}$ is given by
\[
V_{\Delta} \;=\; V_{\emptyset} \,\cup\, \bigcup_{S} \, \bigcup_{j} V(J_{S,j}), \tag{Theorem~\ref{thm:decomposition of Vdelta for k=2}}
\]
where the outer union runs over all admissible nonempty subsets $S \subset [2\ell]$, and for each such $S$, the inner union runs over all values of $j$ specified in~\eqref{maximal j}. The ideals $J_{S,j}$ are as defined in~\eqref{ideal JSj}. 

Moreover, in Corollary~\ref{cor: dimension of V_S^j} we compute the dimensions of the irreducible components:
\[
\Dim(V(J_{S,j})) \;=\; d(2t - 2 - j) + (t - 2)(u + v) + j(\ell - u - v) - \bigl(j^{2} + 2(t - 1 - j)^{2} + 2j(t - 1 - j)\bigr) + \ell,
\]
and in Theorem~\ref{thm:maximum dimension general} we obtain the dimension of $V_{\Delta}$:
\begin{align*}
\begin{cases}
  \max\{d\ell + 2t - \ell - 2,\ (t - 1)(d + \ell - t + 1) + \ell\}& \ell < 2t-2, \\
  \max\{(t - 1)(2d - 2t + \ell + 2),\ (t - 1)(d + \ell - t + 1) + \ell\} & \ell \geq 2t-2.
\end{cases}
\end{align*}

Furthermore, in Proposition~\ref{prop:degree of Vdelta} we express the degree of $V_{\Delta}$ in terms of the degree of $V_{\emptyset}$:
\begin{itemize}
    \item[{\rm (i)}] If $\ell > 2t - 2$, the degree of $V_\Delta$ is
    \[
        \alpha \cdot \mathbbm{1}_{(t - 1)(d - t + 1) \geq \ell} 
        \;+\;
        \deg(V_\emptyset) \cdot \mathbbm{1}_{(t - 1)(d - t + 1) \leq \ell},
    \]
    \item[{\rm (ii)}] If $\ell \leq 2t - 2$, the degree of $V_\Delta$ is
    \[
        \beta \cdot \mathbbm{1}_{(d - t)(\ell - t) + d - 1 \geq \ell}
        \;+\;
        \deg(V_\emptyset) \cdot \mathbbm{1}_{(d - t)(\ell - t) + d - 1 \leq \ell},
    \]
\end{itemize}
where 
\[
\alpha \coloneqq 
\sum_{u = t - 1}^{\ell - t + 1}
\binom{\ell}{u}
\det\!\bigl[\tbinom{d + u - i - j}{\, d - i \,}\bigr]_{i,j = 1}^{t - 1}
\cdot
\det\!\bigl[\tbinom{d + \ell - u - i - j}{\, d - i \,}\bigr]_{i,j = 1}^{t - 1},
\qquad
\beta \coloneqq
\binom{\ell}{t - 1}
\binom{t - 1}{2t - 2 - \ell}
d^{\, 2t - 2 - \ell}.
\]
Here $\mathbbm{1}_{a \geq b} = 1$ if $a \geq b$ and $0$ otherwise.

Finally, in Section~\ref{sec: combinatorial approach to compute degree} we develop a combinatorial method to compute the degree of $V_{\emptyset}$ by counting certain lattice paths. For the particular case $d = t$, we obtain the explicit formula
\[
\deg(V_{\emptyset})
\;=\;
d^{d - 1}(d - 1)^{\ell - d + 1}\binom{\ell}{d-1}.\tag{Corollary~\ref{cor:count}}
\]
For the general case $d\geq t$, we provide the generating function
\begin{equation*}
\mathcal{G}_{d,\ell,t}(z_1,\dotsc,z_\ell) =
\prod_{j=1}^\ell z^{\min(i,t-1)}\; \cdot \;
\det \left( 
h_{d-i}(z_j,z_{j+1},\dotsc,z_\ell)
\right)_{1\leq i,j \leq t-1},
\end{equation*}
where $h_i$ is the complete homogeneous symmetric polynomial of degree $i$, and show that
\[
  \deg(V_{\emptyset})=\left.\frac{\partial}{\partial z_1} \dotsb \frac{\partial}{\partial z_\ell}
  \mathcal{G}_{d,\ell,t}(z_1,\dotsc,z_\ell) \right\vert_{z_1=\dotsb = z_\ell =1}.\tag{Proposition~\ref{cor:total_transversal}}
\]
\end{theoremC}

Theorems~\ref{theoremB} and~\ref{theoremC} have the following statistical interpretation.

\begin{corollary}\label{cor:Iempty int property}
Consider the CI model $\mathcal{C}$ from~\eqref{model}, where $X$, $Y_1$, and $Y_2$ are observed random variables with finite state spaces of sizes $d$, $k$, and $\ell$, respectively, and $H$ is a hidden variable with state space of size $t-1$. 
Assume that $d$, $k$, $\ell$, and $t$ are integers satisfying $k=2$ or $t=\ell$. 
Then the \emph{intersection property} holds; that is,
\[
X \mathrel{\perp\!\!\!\perp} \{Y_1, Y_2\} \mid H.
\]
More precisely, the component $V_{\emptyset}$ in the decomposition of $\mathcal{V}_\Delta$ describes the set of joint distributions with full support (those without structural zeros) in the probability table. 
This component corresponds to the polynomials defining the conditional independence relation above. 
\end{corollary}

\medskip
\noindent
\textbf{Outline of the paper.}
In Section~\ref{sec:intro}, we introduce the basic definitions, fix notation, and summarize the main results.  
Section~\ref{sec:decomposition} develops the general decomposition framework: we introduce the varieties $U_S$ and $V_S$ and show that $V_\Delta$ decomposes as a union of the $V_S$, which we then refine to an irredundant irreducible decomposition.  
In Section~\ref{sec:FS}, we introduce the auxiliary quasi-affine varieties $F_S$ and establish a correspondence between the irreducible components of $V_S$ and those of $F_S$, thereby reducing the decomposition problem to a lower-dimensional setting. Section~\ref{sec:t=l} illustrates this framework by recovering the case $t=\ell$.  
Section~\ref{sec:dim-eq} focuses on the 
defining equations for the irreducible components of $V_\Delta$.  
In Section~\ref{sec:VDelta-k=2}, we specialize to the case $k=2$ and obtain a complete irredundant irreducible decomposition of $V_\Delta$.  
Sections~\ref{sec:dim} and~\ref{sec:degree} compute the dimensions and degrees of the resulting varieties, completing their geometric description.  
Finally, Section~\ref{sec:quasi_product} places our results in a broader matroid-theoretic context by interpreting conditional independence varieties in terms of quasi-products of matroids.

\section{General decomposition theorem}\label{sec:decomposition}

In this section, we address Question~\ref{question decomposition} by developing a general framework for decomposing the variety $V_{\Delta}$. 
The main result, Theorem~\ref{irreducible}, provides an implicit description of the irreducible decomposition of $V_{\Delta}$ by reducing the problem to a more tractable one. 
Specifically, it shows that the decomposition of $V_{\Delta}$ may be understood in terms of the irreducible components of certain auxiliary varieties $U_{S}$ associated with subsets $S \subset [k\ell]$, which are introduced in the next subsection.
This reduction identifies the combinatorial structure underlying $V_{\Delta}$ and forms the basis for all subsequent results.

\subsection{Decomposition into varieties \texorpdfstring{$V_S$}{VS}}

We now introduce the varieties $U_{S}$ and $V_{S}$, defined for subsets $S \subset [k\ell]$, which play a central role in the decomposition of $V_{\Delta}$. 
The subset $S$ encodes which matrix entries are forced to vanish, while the remaining entries are constrained by matroid-theoretic rank conditions arising naturally from the geometry of $V_{\Delta}$.

\begin{notation}
We interpret $V_{\Delta}$ as the variety consisting of all tuples $\gamma = (\gamma_{j} : j \in [k\ell]) \in (\CC^{d})^{k\ell}$ of $k\ell$ vectors in $\CC^{d}$ satisfying the following conditions:
\begin{itemize}
\item For any $i_{1}, i_{2} \in [k\ell]$ belonging to the same column of the matrix $\mathcal{Y}$ in~\eqref{matrix}, the vectors $\{\gamma_{i_{1}}, \gamma_{i_{2}}\}$ are linearly dependent.\vs
\item For any subset $\{i_{1}, \ldots, i_{t}\} \subset [k\ell]$ of size $t$ belonging to the same row of $\mathcal{Y}$, the vectors $\{\gamma_{i_{1}}, \ldots, \gamma_{i_{t}}\}$ are linearly dependent.
\end{itemize}
\end{notation}

\begin{example}\label{example: first example}
Let $k=3$, $\ell=6$, and $t=d=4$. Consider the following matrix:
\begin{equation}\label{matrix with vectors}
\begin{pmatrix}
(0,0,0,0) & (0,0,0,0) & (2,0,0,2) & (2,0,0,4) & (0,3,0,3) & (0,3,0,6)\\
(0,0,1,1) & (0,0,1,2) & (0,0,0,0) & (0,0,0,0) & (0,1,0,1) & (0,1,0,2)\\
(0,0,4,4) & (0,0,4,8) & (1,0,0,1) & (1,0,0,2) & (0,0,0,0) & (0,0,0,0)
\end{pmatrix}.
\end{equation}
This matrix represents an element $\gamma \in V_{\Delta}$, since any two vectors in the same column and any four vectors in the same row are linearly dependent.  
Using the indexing of the entries of the matrix $\mathcal{Y}$ from~\eqref{matrix}, we may equivalently regard $\gamma$ as a tuple of $k\ell$ vectors. For instance, under that indexing we have $\gamma_{7} = (2,0,0,2)$ and $\gamma_{17} = (0,1,0,2)$. 
\end{example}

For $\gamma \in V_{\Delta}$ represented by a $k \times l$ matrix of vectors, we index the vectors as in the entries of the matrix $\mathcal{Y}$ from~\eqref{matrix}. We now associate a matroid on the ground set $[\ell]$ to each point $\gamma \in V_{\Delta}$.

\begin{definition}\label{matr aso}
Let $\gamma = (\gamma_{j} : j \in [k\ell]) \in V_{\Delta}$. 
For each $i \in [\ell]$, define a vector $v_{i} \in \CC^{d}$ by
\[
v_{i} =
\begin{cases}
0, & \text{if } \gamma_{j} = 0 \text{ for all } j \in C_{i}, \\[4pt]
\gamma_{j}, & \text{for any } j \in C_{i} \text{ with } \gamma_{j} \neq 0.
\end{cases}
\]
The matroid $\mathcal{M}_\gamma$ on $[\ell]$ is defined as the matroid represented by the tuple of vectors $(v_{1}, \ldots, v_{\ell})$. 

\smallskip
Note that the tuple $(v_{1}, \ldots, v_{\ell})$ may depend on the choice of the nonzero vector $\gamma_{j} \in C_{i}$.
However, the matroid $\mathcal{M}_\gamma$ is well defined, since for each $i \in [\ell]$, any two nonzero vectors $\gamma_{j_{1}}$ and $\gamma_{j_{2}}$ with $j_{1}, j_{2} \in C_{i}$ are scalar multiples of one another.
\end{definition}

\begin{example}\label{example: matroid Mgamma}
Consider the element $\gamma \in V_{\Delta}$ from Example~\ref{example: first example}.  
Define
\[
(v_{1},v_{2},v_{3},v_{4},v_{5},v_{6})
= ((0,0,1,1), (0,0,1,2), (1,0,0,1), (1,0,0,2), (0,1,0,1), (0,1,0,2)).
\]
It is immediate that for each $j \in C_{i}$, the vector $\gamma_{j}$ is a scalar multiple of $v_{i}$.  
The associated matroid $\mathcal{M}_{\gamma}$ is the one determined by $(v_{1},\ldots,v_{6})$.  
In this case, it is a matroid of rank $4$ on the ground set $[6]$ whose collection of non-spanning circuits is
$\{\{1,2,3,4\},\ \{3,4,5,6\},\ \{1,2,5,6\}\}.$
\end{example}
\begin{notation}
Let $S \subset [k\ell]$. For each $i \in [k]$, we define
$S_{i} = \{\, j \in [\ell] : \mathcal{Y}_{i,j} \notin S \,\} \subset [\ell].$
\end{notation}

\begin{example}
Let $k = 3$, $\ell = 6$, and $S = \{1,4,8,11,15,18\}$. Then
\[
(S_1, S_2, S_3) = \{\{3,4,5,6\},\; \{1,2,5,6\},\; \{1,2,3,4\}\}.
\]
\end{example}
We next associate a variety to each subset $S \subset [k\ell]$.

\begin{definition}\label{def: variety US}
For every subset $S \subset [k\ell]$, define $U_{S} \subset V_{\Delta}$ as the set of all tuples of vectors $\gamma = (\gamma_{j} : j \in [k\ell]) \in V_{\Delta}$ satisfying:
\begin{itemize}
\item $\gamma_{j} = 0$ if and only if $j \in S$; \vs
\item for each $i \in [k]$, the set $S_{i}$ is a flat of rank $t - 1$ in the matroid $\mathcal{M}_\gamma$.
\end{itemize}
Furthermore, let $V_{S}$ denote the Zariski closure of $U_{S}$ in $(\CC^{d})^{k\ell}$.
\end{definition}

\begin{example}\label{example: S1,S2,S3}
Let $k = 3$, $\ell = 6$, and $d = t = 4$, and let $S = \{1,4,8,11,15,18\}$. 
Consider the element $\gamma \in V_{\Delta}$ from Example~\ref{example: first example}. 
We have $\gamma \in U_{S}$ since $\gamma_{j} = 0$ precisely when $j \in S$, and because the sets
\[
(S_{1}, S_{2}, S_{3}) = \{\,\{3,4,5,6\},\; \{1,2,5,6\},\; \{1,2,3,4\}\}
\]
are flats of rank three in the matroid $\mathcal{M}_{\gamma}$, as previously established in 
Example~\ref{example: matroid Mgamma}.
\end{example}

\begin{lemma}\label{quasi}
For every subset $S \subset [k\ell]$, the set $U_{S}$ is a quasi-affine variety in $(\CC^{d})^{k\ell}$.
\end{lemma}

\begin{proof}
A tuple of vectors $\gamma = (\gamma_{j} : j \in [k\ell])$ is in $U_{S}$ if and only if the following conditions hold:
\begin{itemize}
\item[{\rm (i)}] $\gamma \in V_{\Delta}$;
\item[{\rm (ii)}] $\gamma_{j} = 0$ if and only if $j \in S$;
\item[{\rm (iii)}] for each $i \in [k]$ and every $r \in [\ell] \setminus S_{i}$, there exist a $(t-1)$-subset 
$\{i_{1}, \ldots, i_{t-1}\} \subseteq \cup_{j \in S_{i}} C_{j}$ 
and an element $i_{t} \in C_{r}$ such that the vectors 
$\{\gamma_{i_{1}}, \ldots, \gamma_{i_{t}}\}$ are linearly independent.
\end{itemize}
Condition~{\rm (iii)} is equivalent to requiring that each $S_{i}$ is a flat of rank $t - 1$ in the matroid $\mathcal{M}_\gamma$.  
Since $V_{\Delta}$ is Zariski closed, condition~{\rm (ii)} 
imposes vanishing and non-vanishing constraints and condition~{\rm (iii)} imposes polynomial rank constraints, it follows that $U_{S}$ is the intersection of a Zariski-closed and a Zariski-open subset of $(\CC^{d})^{k\ell}$.  
Hence, $U_{S}$ is a quasi-affine variety.
\end{proof}

The main result of this subsection is Theorem~\ref{irreducible}, where we compute the irredundant irreducible decomposition of the variety $V_{\Delta}$, thus providing a partial answer to Question~\ref{question decomposition}. The result is partial in the sense that it reduces the problem to determining the irreducible components of the varieties $U_{S}$, where $S \subset [k\ell]$ is an admissible subset as defined in Definition~\ref{admissible}.

\medskip

In all subsequent proofs, we will use the notion of \emph{infinitesimal motion}s, introduced below.
\begin{definition}
An \emph{infinitesimal motion} (or \emph{perturbation}) refers to a modification that can be made arbitrarily small. 
We say that a \emph{perturbation of $x$ produces an element of a set $X$} if $x$ lies in the 
Euclidean closure of $X$; equivalently, for every $\varepsilon > 0$ there exists a perturbation of $x$ at distance at most $\varepsilon$ that belongs to $X$. 
\end{definition}

\begin{example}\label{example perturbation} We build on \cite[Remark 5.10]{clarke2021matroid}.
Let $k=2,\ell=5$ and $d=t=3$. Consider the tuple of vectors 
$\gamma=(\gamma_{j}:j\in [10])$ represented by the matrix
\[
\begin{pmatrix}
(1,0,0) & (1,1,0) & (0,0,0) & (0,0,0) & (0,0,0)\\
(0,0,0) & (0,0,0) & (0,0,0) & (0,1,1) & (0,0,1)
\end{pmatrix}.
\]
Clearly, $\gamma\in V_{\Delta}$, since any two vectors in the same column and any three vectors in the same row are linearly dependent.
Choosing one representative from each column as in Definition~\ref{matr aso}, we set
\[
(v_{1},v_{2},v_{3},v_{4},v_{5})=((1,0,0),(1,1,0),(0,0,0),(0,1,1),(0,0,1)).
\]
The associated matroid $\mathcal{M}_{\gamma}$ on $[5]$ has rank $3$, with its only nonspanning circuit being $\{3\}$, a loop. These vectors are represented by the matrix
$A=\scalebox{0.85}{$\begin{bmatrix}
1 & 1 & 0 & 0 & 0 \\
0 & 1 & 0 & 1 & 0 \\
0 & 0 & 0 & 1 & 1
\end{bmatrix}$}$.
Consider the subset $S=\{2,4,7,9\}\subset [10]$, which determines $(S_{1},S_{2})=\{\{1,2,3\},\{3,4,5\}\}$. Note that $\gamma\notin U_{S}$, since $\gamma_{5}=\gamma_{6}=0$ while $5,6\notin S$. However, we can show that $\gamma\in V_{S}=\overline{U_{S}}$ by applying an arbitrarily small perturbation to $\gamma$ that produces an element of $U_{S}$.
Specifically, Figure~\ref{fig: perturbation example} illustrates perturbing $A$ to
\[
A(\epsilon) = 
\begin{bmatrix}
1 & 1 & 0 & 0 & 0 \\
0 & 1 & \epsilon & 1 & 0 \\
0 & 0 & 0 & 1 & 1 
\end{bmatrix},
\]
which corresponds to moving $v_{3}$ so that it is non-zero and lies at the intersection of the planes spanned by $\{v_{1},v_{2}\}$ and $\{v_{4},v_{5}\}$. This induces the perturbed tuple
\[
\gamma(\epsilon)=
\begin{pmatrix}
(1,0,0) & (1,1,0) & (0,\epsilon,0) & (0,0,0) & (0,0,0)\\
(0,0,0) & (0,0,0) & (0,\epsilon,0) & (0,1,1) & (0,0,1)
\end{pmatrix}.
\]

It is now straightforward to verify that $\gamma(\epsilon)\in U_{S}$. Since $\gamma(\epsilon)$ can be made arbitrarily close to $\gamma$, we conclude $\gamma\in \overline{U_{S}}=V_{S}$. In what follows, we will refer to such motions simply as \emph{perturbations}.

\begin{figure}[!ht]
    \centering
    \includegraphics[scale=0.85,page=3]{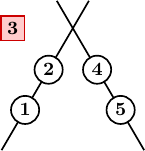}
    \quad
    \includegraphics[scale=0.85,page=5]{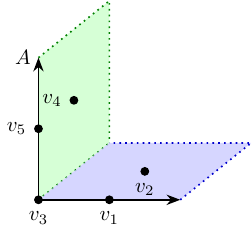}
    \quad
    \includegraphics[scale=0.85,page=6]{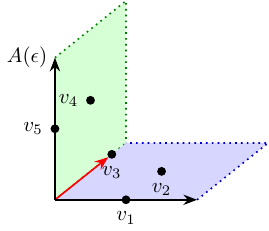}
    \quad
    \includegraphics[scale=0.85,page=4]{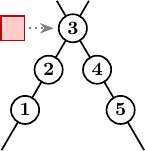}
    \caption{Illustration of the perturbation of the vector $v_{3}$ from Example~\ref{example perturbation}.  
    }
    \label{fig: perturbation example}
\end{figure}
\end{example}

\begin{definition}\label{admissible}
A subset $S \subset [k\ell]$ is called \emph{admissible} if it satisfies the following two conditions:
\begin{itemize}
\item for every $i \in [k]$, we have $\lvert R_{i} \setminus S \rvert \geq t - 1$;
\item for every $j \in [\ell]$, the column set $C_{j}$ is not entirely contained in $S$.
\end{itemize}

\begin{example}
Let $k=3$, $\ell=6$, and $d=t=4$, and consider the subset $S=\{1,4,8,11,15,18\}$ as in Example~\ref{example: S1,S2,S3}. For each $i\in [3]$, we have $\lvert R_{i}\setminus S\rvert = 4 \ge t-1 = 3$. Moreover, no column set $C_{j}$ is entirely contained in $S$. It follows that $S$ is admissible.
\end{example}

For each admissible~set~$S \subset [k\ell]$, let 
$U_{S,1}, \ldots, U_{S,a_{S}}$
be the irreducible components of $U_{S}$.
We denote the corresponding irreducible components of $V_{S}$ by 
\[
V_{S,1}, \ldots, V_{S,a_{S}}, \quad \text{where $V_{S,i} = \overline{U_{S,i}}$}.
\]
\end{definition}

We now give the irredundant irreducible decomposition of the variety $V_{\Delta}$.

 \begin{theorem}\label{irreducible}
The variety $V_{\Delta}$ has the following irredundant irreducible decomposition:
\begin{equation}\label{inclu 2}
V_{\Delta} = \bigcup_{S} \bigcup_{i=1}^{a_{S}} V_{S,i},
\end{equation}
where the union ranges over all admissible subsets $S \subset [k\ell]$ from Definition~\ref{admissible} such that $V_{S}\neq \emptyset$.
\end{theorem}

\begin{proof}
We will first show that $V_{\Delta} = \cup_{S} V_{S}$. By definition, each $U_{S}$ is contained in $V_{\Delta}$, hence $V_{S} = \overline{U_{S}} \subset V_{\Delta}$. 
To prove the reverse inclusion, let $\gamma = (\gamma_{j} : j \in [k\ell])$ be a tuple of vectors in $V_{\Delta}$. 
We show that $\gamma$ lies in $V_S=\overline{U_{S}}$ for some admissible subset $S$. 
This is achieved by perturbing the vectors of $\gamma$ to obtain a new tuple $\widetilde{\gamma}$ such that $\widetilde{\gamma} \in U_S$ for some admissible $S$. 
The construction of $\widetilde{\gamma}$ is as follows.

\medskip
\noindent\textbf{Step~1.} Suppose there exists $i \in [\ell]$ such that $\gamma_{j} = 0$ for all $j \in C_{i}$.  
Consider the matrix entry $ki = \mathcal{Y}_{k,i}$ from~\eqref{matrix}.  
Since $\gamma \in V_{\Delta}$, the set of vectors $\{\gamma_{j} : j \in R_{k}\}$ spans a subspace of dimension at most $t-1$.  
We may perturb the zero vector $\gamma_{ki} = 0$, replacing it with a nonzero vector in this span.  
The resulting configuration remains in $V_{\Delta}$ and ensures that not all vectors in column $C_{i}$ are zero.  
By iterating this perturbation across all columns, we obtain a tuple $\gamma^{1} \in V_{\Delta}$ with this property for every $i \in [\ell]$.

\medskip
\noindent\textbf{Step~2.} We next perturb $\gamma^{1}$ to obtain a tuple $\gamma^{2} \in V_{\Delta}$ satisfying $\rank(\gamma^{2}) \geq t-1$.  
The perturbation is chosen so as not to introduce loops, ensuring that for every $i \in [\ell]$, the corresponding column $C_{i}$ is not entirely zero.

\medskip
\noindent\textbf{Step~3.} Let $M_{\gamma^{2}}$ denote the matroid associated with the collection $\{v_{1}, \ldots, v_{\ell}\}$ as in Definition~\ref{matr aso}, and let $\operatorname{rk}$ be its rank function.  
Define $S = \{ j \in [k\ell] : \gamma^{2}_{j} = 0 \}$.  
Then:
\begin{itemize}
\item For each $j \in [\ell]$, the vectors $\{\gamma^{2}_{p} : p \in C_{j}\}$ are scalar multiples of $v_{j}$.
\item Consequently, for each $i \in [k]$, we have $\operatorname{rk}\{\gamma^{2}_{p} : p \in R_{i}\} = \operatorname{rk}(S_{i})$.
\end{itemize}

Suppose that for some $i \in [k]$, the set $S_{i}$ is not a flat of rank $t-1$ in $M_{\gamma^{2}}$.  
Then there exists $r \in [\ell] \setminus S_{i}$ such that $\operatorname{rk}(S_{i} \cup \{r\}) \leq t-1$.  
Consider the entry $i + k(r-1) = \mathcal{Y}_{i,r}$ in~\eqref{matrix}.  
Since $r \notin S_{i}$, we have $i + k(r-1) \in S$, implying $\gamma^{2}_{i+k(r-1)} = 0$.  
The condition $\operatorname{rk}(S_{i} \cup \{r\}) \leq t-1$ gives
\[
\operatorname{rk}\{v_{j} : j \in S_{i} \cup \{r\}\} \leq t-1.
\]
We then perturb the zero vector $\gamma^{2}_{i+k(r-1)}$, replacing it with a nonzero scalar multiple of $v_{r}$.  
Since all vectors in $C_{r}$ are scalar multiples of $v_{r}$, the rank of column $C_{r}$ remains one, while the rank of row $R_{i}$ stays at most $t-1$.  
Thus, the perturbed tuple still belongs to $V_{\Delta}$, now with $\gamma^{2}_{i+k(r-1)} \neq 0$.

Iterating this process, we obtain $\gamma^{3} \in V_{\Delta}$ such that, denoting $S' = \{ j \in [k\ell] : \gamma^{3}_{j} = 0 \}$, each $S'_{i}$ is a flat of rank $t-1$ in $M_{\gamma^{3}}$.

\medskip
\noindent\textbf{Admissibility of $S'$.} 
We now prove that the subset $S'$ is admissible.  
To see this, note that:
\begin{itemize}
\item Since each $S'_{i}$ is a flat of rank $t-1$, we have $|R_{i} \setminus S'| = |S'_{i}| \geq t-1$ for all $i \in [k]$.
\item For each $i \in [\ell]$, at least one entry in the column $C_{i}$ is nonzero, hence no $C_{i}$ is fully contained in~$S'$.
\end{itemize}

Finally, since $\gamma^{3}_{j} = 0$ if and only if $j \in S'$ and $S'_{i}$ is a flat of rank $t-1$ in $M_{\gamma^{3}}$ for every $i \in [k]$, we conclude that $\gamma^{3} \in U_{S'}$, as desired. Hence, $V_{\Delta} = \cup_{S} V_{S}$.

\medskip
\noindent{\bf Irredundant decomposition.} To prove that the decomposition in \eqref{inclu 2} is irredundant, assume for contradiction, that there exist two distinct admissible subsets $S^{(1)}, S^{(2)} \subset [k\ell]$ such that 
\[
V_{S^{(1)},a} \subset V_{S^{(2)},b}
\quad \text{and} \quad
V_{S^{(1)},a} \neq \emptyset.
\]
By Lemma~\ref{quasi}, the variety $U_{S^{(1)}}$ is quasi-affine, and thus every irreducible component of $V_{S^{(1)}}$ intersects $U_{S^{(1)}}$ nontrivially. Consequently, we may choose a tuple of vectors 
\[
\gamma = (\gamma_{j} : j \in [k\ell]) \in V_{S^{(1)},a} \cap U_{S^{(1)}}.
\]
Since $\gamma \in V_{S^{(1)},a} \subset V_{S^{(2)},b}$, it follows that $\gamma_{j} = 0$ for all $j \in S^{(2)}$. On the other hand, as $\gamma \in U_{S^{(1)}}$, we have $\gamma_{j} = 0$ if and only if $j \in S^{(1)}$. Hence, we must have $S^{(1)} \supsetneq S^{(2)}$.

Now, since $S^{(1)} \supsetneq S^{(2)}$, there exists $p \in [k]$ such that 
$R_{p} \setminus S^{(1)} \subsetneq R_{p} \setminus S^{(2)}$,
or equivalently, ${S^{(1)}}_{p} \subsetneq {S^{(2)}}_{p}$, where $R_{p}$ denotes the $p^{\text{th}}$ row of the matrix $\mathcal{Y}$ from~\eqref{matrix}.  
Because $\gamma \in V_{S^{(1)},a} \subset V_{S^{(2)},b}$, the set ${S^{(2)}}_{p}$ has rank at most $t-1$ in $\mathcal{M}_\gamma$. However, since $\gamma \in U_{S^{(1)}}$, ${S^{(1)}}_{p}$ is a flat of rank $t-1$ in $\mathcal{M}_\gamma$. This is impossible, as ${S^{(1)}}_{p} \subsetneq {S^{(2)}}_{p}$ and $\rank({S^{(2)}}_{p}) \le t-1$, hence the decomposition in~\eqref{inclu 2} is indeed irredundant.
\end{proof}

Theorem~\ref{irreducible} reduces the problem of finding the irreducible decomposition of $V_{\Delta}$ to the following question.
\begin{question}\label{quest us}
Given an admissible subset $S \subset [k\ell]$, determine the irreducible decomposition of~$U_{S}$.
\end{question}

\subsection{Decomposition into varieties \texorpdfstring{$F_{S}$}{FS}}\label{sec:FS}

As shown in Theorem~\ref{irreducible}, determining the irreducible components of $V_{\Delta}$ reduces to finding the irreducible decompositions of the varieties $U_{S}$ for admissible subsets $S \subset [k\ell]$. 
We now further simplify this task by introducing auxiliary varieties $F_{S} \subset (\CC^{d})^{\ell}$ of lower dimension, whose irreducible components are in direct correspondence with those of $U_{S}$; see Theorem~\ref{corresp}. This correspondence allows all geometric questions about $V_{\Delta}$ to be addressed in the simpler setting of the varieties $F_S$.

\begin{definition}\label{fs}
For a tuple of vectors $\gamma = (\gamma_{1}, \ldots, \gamma_{n})$ in a common vector space, we denote by $\mathcal{N}_{\gamma}$ the matroid on $[n]$ determined by these vectors. Let $S \subset [k\ell]$ be an admissible subset. We define $F_{S} \subset (\CC^{d})^{\ell}$ as the set of all tuples of vectors $\gamma = (\gamma_{i} : i \in [\ell]) \in (\CC^{d})^{\ell}$ satisfying:
\begin{itemize}
\item $\gamma_{i} \neq 0$ for all $i \in [\ell]$;\vs
\item for each $i \in [k]$, the set $S_{i}$ is a flat of rank $t-1$ in $\mathcal{N}_{\gamma}$.
\end{itemize}
\end{definition}

By an argument analogous to that of Lemma~\ref{quasi}, we have:
\begin{lemma}
For every subset $S\subset [k\ell]$, the set $F_{S}$ is a quasi-affine variety in $\CC^{d\ell}$.
\end{lemma}

\begin{example}\label{example: FS}
Let $k=3$, $\ell=6$, and $d=t=4$, and consider the subset $S=\{1,4,8,11,15,18\}$ and the tuple $\gamma\in V_{\Delta}$ as in Example~\ref{example: S1,S2,S3}. In this instance, we have $\gamma\in U_{S}$. 

The associated matroid $M_{\gamma}$ is obtained by selecting one representative vector from each column of the matrix in~\eqref{matrix with vectors}. Here, we may choose
\[
(v_{1},v_{2},v_{3},v_{4},v_{5},v_{6})
= ((0,0,1,1), (0,0,1,2), (1,0,0,1), (1,0,0,2), (0,1,0,1), (0,1,0,2)),
\]
as described in the construction of $M_{\gamma}$. It is straightforward to verify that $(v_{1},\dots,v_{6})\in F_{S}$.
\end{example}

As Example~\ref{example: FS} illustrates, any element $\gamma\in U_{S}$ naturally determines a set of elements in $F_{S}$, all obtained by choosing one representative (up to a nonzero scalar) from each column $C_{j}$. Conversely, as we will see in Theorem\ref{corresp}, every element of $F_{S}$ can be lifted back to an element of $U_{S}$ via a polynomial map
\[
F_{S} \times (\mathbb{C}^{\ast})^{|[k\ell]\setminus S|} \longrightarrow U_{S}.
\]
Note that each tuple of vectors in $U_{S}$ is indexed by the elements of $[k\ell]$, corresponding to the entries of the matrix $\mathcal{Y}$ from~\eqref{matrix}, whereas in $F_{S}$ the vectors are indexed by $[\ell]$. This results in a substantial reduction in the dimension of the ambient space. 

\begin{notation}
For a set $A$, we denote by $\CC^{A}$ the set of all collections of complex numbers $(\lambda_{a})_{a \in A}$ indexed by the elements of $A$.
\end{notation}

The following theorem establishes the desired connection between the varieties $U_{S}$ and $F_{S}$.

\begin{theorem}\label{corresp}
The irreducible components of $U_{S}$ are in one-to-one correspondence with those of $F_{S}$.
\end{theorem}

\begin{proof}
Consider the polynomial map
\begin{equation}\label{psi}
\psi : (\CC^{\ast})^{|[k\ell]\setminus S|}\times F_{S}\longrightarrow U_{S}, \qquad 
\psi((\lambda_{i})_{i\in [k\ell]\setminus S},\gamma)_{p} =
\begin{cases}
0, & \text{if } p\in S,\\[4pt]
\lambda_{p}\gamma_{j}, & \text{if } p\in C_{j}\setminus S,
\end{cases}
\end{equation}
where the indices $j$ and $p$ correspond to the positions of the respective vectors in their tuples.  
By the definitions of $U_{S}$ and $F_{S}$, the map $\psi$ is well-defined, its image lies in $U_{S}$, and it is surjective.

Let $F_{S,1},\ldots,F_{S,b_{S}}$ denote the irreducible components of $F_{S}$, and for simplicity set  
\[
F_{i} \coloneqq (\CC^{\ast})^{|[k\ell]\setminus S|}\times F_{S,i}, \quad i\in [b_{S}].
\]
We claim that $U_{S}$ admits the irreducible decomposition
$U_{S} = \bigcup_{i=1}^{b_{S}}\psi(F_{i})$,
or equivalently,
\begin{equation}\label{vs}
V_{S} = \bigcup_{i=1}^{b_{S}}\overline{\psi(F_{i})}.
\end{equation}
The equality in~\eqref{vs} follows from the surjectivity of $\psi$. Each $F_{S,i}$ is irreducible, and since the product of irreducible varieties is irreducible, each $F_{i}$ is irreducible as well. Moreover, $\psi$ being a polynomial map is continuous, and the image of an irreducible set under a continuous map is again irreducible. Hence, each $\psi(F_{i})$ is irreducible, and every component in~\eqref{vs} is irreducible. To complete the proof, it remains to verify that this decomposition is non-redundant.

Assume, for the contrary, that there exist distinct indices $i,j\in [b_{S}]$ such that  
\begin{equation}\label{clos}
\overline{\psi(F_{i})}\subset \overline{\psi(F_{j})}.
\end{equation}
Let $\pi_{2}$ denote the projection onto the second factor $F_{S}$ in the product appearing in~\eqref{psi}. We claim that
\[
\pi_{2}\bigl(\psi^{-1}(\overline{\psi(F_{j})})\bigr)\subset \overline{F_{S,j}}.
\]
Suppose that $\psi(\lambda,\gamma)\in \overline{\psi(F_{j})}$, with $\lambda\in (\CC^{\ast})^{|[k\ell]\setminus S|}$ and $\gamma\in F_{S}$.  
Since $\gamma=\pi_{2}(\psi^{-1}(\psi(\lambda,\gamma)))$, it suffices to show that $\gamma\in \overline{F_{S,j}}$.  
Let $\tau=\psi(\lambda,\gamma)\in U_{S}$. For each $i\in [\ell]$, pick $p_{i}\in [k\ell]\setminus S$, so that each vector $\tau_{p_{i}}$ is a nonzero scalar multiple of $\gamma_{i}$. Thus, the tuples $\gamma=(\gamma_{1},\ldots,\gamma_{\ell})$ and $(\tau_{p_{1}},\ldots,\tau_{p_{\ell}})$ differ only by rescaling.

Since $F_{S,j}$ is quasi-affine, so is $\psi(F_{j})$; hence its Zariski closure coincides with its Euclidean closure. It follows that $\tau\in \overline{\psi(F_{j})}$ can be approximated by points of $\psi(F_{j})$ via arbitrarily small perturbations of its coordinates.  
Moreover, $\psi(F_{j})$ can be described as
\[
\psi(F_{j}) = \{\kappa\in V_{\Delta} : M_{\kappa}\in F_{S,j}\}.
\]
Thus, the vectors of $\tau$ can be slightly perturbed to obtain a tuple in $\psi(F_{j})$, implying $(\tau_{p_{1}},\ldots,\tau_{p_{\ell}})\in \overline{F_{S,j}}$.  
Since this tuple differs from $\gamma$ only by nonzero scalars, it follows that $\gamma\in \overline{F_{S,j}}$, as desired.

Now note that $F_{S,i}\subset \pi_{2}\bigl(\psi^{-1}(\psi(F_{i}))\bigr)$.  
Applying $\pi_{2}\circ \psi^{-1}$ to both sides of~\eqref{clos}, we obtain
\[
F_{S,i}\subset \pi_{2}\bigl(\psi^{-1}(\overline{\psi(F_{i})})\bigr)
\subset \pi_{2}\bigl(\psi^{-1}(\overline{\psi(F_{j})})\bigr)
\subset \overline{F_{S,j}}.
\]
Hence $\overline{F_{S,i}}\subset \overline{F_{S,j}}$, contradicting the fact that $F_{S,i}$ and $F_{S,j}$ are distinct components of $F_{S}$.  
This contradiction shows that the decomposition in~\eqref{vs} is irredundant, completing the proof.
\end{proof}

\begin{notation}\label{not corres}
In accordance with the correspondence established in Theorem~\ref{corresp}, we denote by $F_{S,i}$ the irreducible component of $F_{S}$ corresponding to $U_{S,i}$, for each $i \in [a_{S}]$.
\end{notation}

Theorem~\ref{corresp} shows that Question~\ref{quest us} reduces to the following.
\begin{question}\label{quest fs}
Given an admissible subset $S \subset [k\ell]$, determine the irreducible decomposition of~$F_{S}$.
\end{question}

Thus, determining the irreducible components of $V_{\Delta}$ reduces to describing the irreducible decompositions of the varieties $F_{S}$. 
In the next two subsections, we carry this out explicitly for two families: the case $t=\ell$ in Section~\ref{sec:t=l} and the case $k=2$ in Section~\ref{sec:VDelta-k=2}.

\subsubsection{Dimension of the components of $V_{\Delta}$ for arbitrary $k$}\label{sec:dim arbitrary k}

As an immediate consequence of the proof of Theorem~\ref{corresp}, we can compute the 
dimensions of the varieties $V_{S,i}$ appearing in the decomposition in~\eqref{psi} in terms of the associated varieties $F_{S,i}$.

\begin{corollary}\label{dimensions}
Let $S \subset [k\ell]$ be an admissible subset. Then, for each $i \in [a_{S}]$, 
\begin{equation}\label{vsj}
\Dim(V_{S,i}) = \Dim(U_{S,i}) = \Dim(F_{S,i}) + \ell(k-1) - \lvert S \rvert.
\end{equation}
\end{corollary}
\begin{proof}
Consider the map $\psi$ defined in~\eqref{psi}. For each $i \in [a_{S}]$, set
$F_{i} := (\CC^{\ast})^{\lvert [k\ell]\setminus S \rvert} \times F_{S,i}$.
By the proof of Theorem~\ref{corresp}, we have $U_{S,i} = \psi(F_{i})$. Since $U_{S,i}$ is quasi-affine, it has the same dimension as its Zariski closure, which gives $\Dim(V_{S,i}) = \Dim(U_{S,i})$.
Moreover, $\Dim(F_{i}) = \Dim(F_{S,i}) + k\ell - \lvert S \rvert$. For any $\gamma \in U_{S,i}$, the fiber $\psi^{-1}(\gamma)$ is isomorphic to $(\CC^{\ast})^{\ell}$. Hence
$\Dim(U_{S,i}) = \Dim(F_{i}) - \ell = \Dim(F_{S,i}) + \ell(k-1) - \lvert S \rvert$.
\end{proof}

\begin{example}
Under the setting of Example~\ref{example: S1,S2,S3}, we have $\Dim(V_{S,i})=\Dim(F_{S,i})+6$.
\end{example}

\subsection{Example: Irreducible decomposition of \texorpdfstring{$V_{\Delta}$}{VDelta} for 
\texorpdfstring{$t=\ell$}{t=l}}\label{sec:t=l}

The irreducible components of $V_{\Delta}$ for $t=\ell$ were determined in \cite{clarke2020conditional}.  
Here, we rederive these results using our framework based on the varieties $F_{S}$ and $U_{S}$, yielding substantially simpler proofs.  
Throughout, we fix $t=\ell$ and first identify the admissible subsets $S\subset [k\ell]$ for which $F_{S}\neq \emptyset$.

\begin{lemma}\label{only}
Let $S \subset [k\ell]$ be admissible with $F_{S}\neq \emptyset$.  
Then either $S=\emptyset$, or $S$ satisfies
\[
\lvert S \cap R_i \rvert = 1 \quad \text{for all } i \in [k]\quad\text{and}\quad C_j \not\subset S \quad \text{for all } j \in [\ell].
\]
\end{lemma}

\begin{proof}
By Definition~\ref{admissible}, a subset $S\subset [k\ell]$ is admissible if
$\lvert R_i \setminus S\rvert \ge t-1$ for all $i\in [k]$ and $C_j\not\subset S$ for all $j\in [\ell]$.
Since $\lvert R_i\rvert=\ell=t$, this is equivalent to
$\lvert S\cap R_i\rvert \le 1$ for all $i\in [k]$ and $C_j\not\subset S$ for all $j\in [\ell]$.
If $\lvert S\cap R_i\rvert=1$ for every $i$, then $S$ satisfies these conditions.

Now suppose there exists $i\in [k]$ with $S\cap R_i=\emptyset$.
As $F_S\neq \emptyset$, take $\gamma\in F_S$.
By Definition~\ref{fs}, each $S_j$ is a flat of rank $t-1$ in the matroid $\mathcal{N}_\gamma$.
The condition $S\cap R_i=\emptyset$ implies $S_i=[\ell]$, so $[\ell]$ is a flat of rank $t-1$ in $\mathcal{N}_\gamma$.
Since $[\ell]$ is the unique flat of that rank, $S_j=[\ell]$ for all $j\in [k]$, and hence $S=\emptyset$.
\end{proof}

\begin{notation}
We define the following collection of subsets of $[k\ell]$:
\[\mathcal{R}=\{S\subset [k\ell]:\ \text{$\size{R_{i}\cap S}=1$ for all $i\in [k]$ and $C_{j}\not \subset S$ for all $j\in [\ell]$}\}.\]
\end{notation}
We show that for $S=\emptyset$ or $S\in\mathcal{R}$, 
the varieties $F_S$ are 
and irreducible, and compute their dimensions.

\begin{proposition}\label{prop:t=l-components}
Let $t=\ell$. Then the following hold:
\begin{itemize}
\item[{\rm (i)}] The variety $F_{\emptyset}$ is nonempty and irreducible, and
$\Dim(V_{\emptyset})=\ell(k+d)-d-1.$
\item[{\rm (ii)}] For each $S\in\mathcal{R}$, the variety $F_{S}$ is nonempty and irreducible, and
$\Dim(V_{S})=\ell(k+d-1)-k.$
\end{itemize}
\end{proposition}

\begin{proof}
(i) By Definition~\ref{fs}, the variety $F_{\emptyset}$ consists of all
$\gamma=(\gamma_{1},\ldots,\gamma_{\ell})\in(\CC^{d})^{\ell}$ with
$\gamma_{i}\neq 0$ for all $i$ and $\rank\{\gamma_{1},\ldots,\gamma_{\ell}\}=\ell-1$.
Thus $F_{\emptyset}$ is a nonempty Zariski open subset of
\[
W=\{\gamma=(\gamma_{1},\ldots,\gamma_{\ell})\in (\CC^{d})^{\ell} : 
\text{$\gamma_{1},\ldots,\gamma_{\ell}$ are linearly dependent}\}.
\]
which is irreducible. Hence $F_{\emptyset}$ is irreducible. Since
$\Dim(W)=d\ell-(d+1-\ell)$, we have $\Dim(F_{\emptyset})=d\ell-(d+1-\ell)$, and
Corollary~\ref{dimensions} yields
$\Dim(V_{\emptyset})=\Dim(F_{\emptyset})+\ell(k-1)=\ell(k+d)-d-1$.

\smallskip
(ii) By Definition~\ref{fs}, 
$F_{S}$ consists of all tuples of vectors 
$\gamma=(\gamma_{1},\ldots,\gamma_{\ell})\in (\CC^{d})^{\ell}$ satisfying:
\begin{center} $\gamma_{i}\neq 0$ for all $i\in [\ell]$, \quad
$\rank\{\gamma_{1},\ldots,\gamma_{\ell}\}=\ell$, \quad and\quad
$\rank\{\gamma_{p}:p\in S_{i}\}=\ell-1$ for every $i\in [k]$.
\end{center} 
Since each $S_{i}$ has size $\ell-1$,
$F_{S}$ is a nonempty Zariski open subset of $(\CC^{d})^{\ell}$, and hence irreducible.
Thus $\Dim(F_{S})=d\ell$, and Proposition~\ref{dimensions} gives
$\Dim(V_{S})=\Dim(F_{S})+\ell(k-1)-k
= \ell(k+d-1)-k.$
\end{proof}

We can now deduce the irredundant irreducible decomposition of $V_{\Delta}$ in the case $t=\ell$.

\begin{corollary}\label{desc del teo}
The irredundant irreducible decomposition of $V_{\Delta}$ is
$V_{\Delta} \;=\; V_{\emptyset} \,\cup\, \bigcup_{S\in \mathcal{R}} V_{S}.$
\end{corollary}

\begin{proof}
The statement follows from Theorems~\ref{irreducible} and~\ref{corresp} together with Proposition~\ref{prop:t=l-components}.
\end{proof}
\begin{example}
Let $k=3$ and $\ell=6$. Then $|\mathcal{R}|=6^{3}-6$, and $V_{\Delta}$ has $6^{3}-5=211$ irreducible components.
\end{example}

\section{Defining equations of the components of \texorpdfstring{$V_{\Delta}$}{VDelta}}\label{sec:dim-eq}

In the previous section, we reduced the irreducible decomposition of $V_{\Delta}$ and the dimensions of its components to the study of the varieties $F_S$. In this section, we show that the defining equations of the components of $V_{\Delta}$ can likewise be recovered from the corresponding data for the components of $F_S$.
\subsection{Defining equations}

We turn to the defining equations of the irreducible components of $V_{\Delta}$. 
For a component $V_{S,i}$, our goal is to describe a complete set of defining equations, i.e.\ generators of its ideal up to radical. 
We show that these equations can be obtained directly from defining equations for the corresponding variety $F_{S,i}$.

\begin{notation} We begin by introducing some notation.
\begin{itemize}
\item Let $X=(x_{i,j})$ denote a $d\times k\ell$ matrix of indeterminates.   
For each $j\in [k\ell]$, we write $X_{j}$ for the $j^{\text{th}}$ column of $X$.

\item Let $Y=(y_{i,j})$ be a $d\times \ell$ matrix of indeterminates, with columns indexed by the elements of $[\ell]$.  
For each $j\in [\ell]$, we denote by $Y_{j}$ the $j^{\text{th}}$ column of $Y$.

\item When referring to the defining equations of a quasi-affine variety $A$, we mean the defining equations of its Zariski closure $\overline{A}$.  
These are precisely the polynomials $p$ that vanish on all points of $A$, or equivalently, the elements of $I(A)$.

\item A \emph{complete set of defining equations} for $A$ refers to a set of generators of $I(A)$ up to radical; that is, a collection of polynomials $\{p_{1},\ldots,p_{r}\}$ satisfying  
$V(\{p_{1},\ldots,p_{r}\})=\overline{A}$.
\end{itemize}
\end{notation}

We next describe how to obtain polynomials in $I(V_{S,i})$ from those in $I(F_{S,i})$, after fixing~some~notation.

\begin{notation}
For each admissible subset $S$ and each $i\in [a_{S}]$, we adopt the following conventions:
\begin{itemize}
\item Every $p\in I(V_{S,i})$ is viewed as a polynomial in $\CC[X]$, where $X=(x_{i,j})$ is a $d\times k\ell$ matrix of indeterminates.
\item Every $p\in I(F_{S,i})$ is viewed as a polynomial in $\CC[Y]$, where $Y=(y_{i,j})$ is a $d\times \ell$ matrix of indeterminates, as above.
\item A subset $T=\{t_{1},\ldots,t_{\ell}\}\subset [k\ell]$ is called an \emph{$S$-representative} if $t_{i}\in C_{i}\setminus S$ for all $i\in [\ell]$.  
We denote by $\mathcal{T}(S)$ the collection of all $S$-representatives.
\end{itemize}
\end{notation}

For each $p\in \CC[Y]$ and each $S$-representative $T$, we now associate a corresponding polynomial $p_{T}\in \CC[X]$.

\begin{definition}
Let $p\in \CC[Y]$ and let $T=\{t_{1},\ldots,t_{\ell}\}\subset [k\ell]$ be an $S$-representative.  
We define the polynomial $p_{T}\in \CC[X]$ via the algebra morphism 
\[
\phi_{T}:\CC[Y]\longrightarrow \CC[X], \qquad y_{i,j}\longmapsto x_{i,t_{j}} \ (i\in[d],\,j\in[\ell]),
\]
and set $p_{T}=\phi_{T}(p)$.
Moreover, for $\gamma=(\gamma_{j}: j\in [k\ell])\in V_{\Delta}$ and an $S$-representative $T=\{t_{1},\ldots,t_{\ell}\}$, set
\[
\gamma_{T} \coloneqq (\gamma_{t_{1}},\ldots,\gamma_{t_{\ell}}).
\]
\end{definition}

\begin{example}
Let $k=3$, $l=6$, and $S=\{1,4,8,11,15,18\}$ as in Example~\ref{example: S1,S2,S3}. In this case, the subset $T = \{2,5,7,12,13,17\} \subset [18]$ is an $S$-representative. The associated 
morphism 
$\phi_{T}:\CC[Y] \longrightarrow \CC[X]$
is given by:
$
y_{i,1} \mapsto x_{i,2}, \quad
y_{i,2} \mapsto x_{i,5}, \quad
y_{i,3} \mapsto x_{i,7}, \quad
y_{i,4} \mapsto x_{i,12}, \quad
y_{i,5} \mapsto x_{i,13}, \quad
y_{i,6} \mapsto x_{i,17}.
$
\end{example}

\begin{lemma}\label{gama t}
Let $\gamma\in U_{S,i}$ and let $T\subset [k\ell]$ be an $S$-representative.  
Then $\gamma_{T}\in F_{S,i}$.
\end{lemma}

\begin{proof}
Let $\psi$ be the polynomial map defined in~\eqref{psi}, and set
$F_i \coloneqq (\CC^{\ast})^{|[k\ell]\setminus S|}\times F_{S,i}$.
From the proof of Theorem~\ref{corresp}, we have $U_{S,i}=\psi(F_i)$.
Suppose $\gamma=\psi(\lambda,\tau)$ with $\lambda\in (\CC^{\ast})^{|[k\ell]\setminus S|}$ and $\tau\in F_{S,i}$.
By definition of an $S$-representative $T=\{t_1,\ldots,t_\ell\}$, each $t_j\in C_j\setminus S$.
Hence, by the definition of $\psi$, the coordinate $\gamma_{t_j}$ is a nonzero scalar multiple of $\tau_j$ for all $j\in[\ell]$.
Thus, $\gamma_T$ and $\tau$ differ only by nonzero scalar multiples.
Since $F_{S,i}$ is closed under multiplication by nonzero scalars, it follows that $\gamma_T\in F_{S,i}$.
\end{proof}

\begin{proposition}\label{prop p}
Let $p\in I(F_{S,i})$ and let $T\subset [k\ell]$ be an $S$-representative.  
Then $p_{T}\in I(V_{S,i})$.
\end{proposition}

\begin{proof}
It suffices to show that $p_T$ vanishes on every element of $U_{S,i}$.
Let $\gamma\in U_{S,i}$.
Since the support of $p_T$ involves only variables indexed by $T$, we have
$p_T(\gamma)=p(\gamma_T)=0$,
where the final equality follows from $p\in I(F_{S,i})$ and Lemma~\ref{gama t}, which implies that $\gamma_T\in F_{S,i}$, and completes the proof.
\end{proof}

We now show that defining equations for $F_{S,i}$ induce those for $V_{S,i}$.

\begin{proposition}\label{equations}
Let $\{p_{1},\ldots,p_{r}\}\subset \CC[Y]$ be a complete set of defining equations for $F_{S,i}$.  
Then the following set of polynomials in $\CC[X]$ forms a complete set of defining equations for $V_{S,i}$:
\begin{equation}\label{ecu}
\{x_{i,j}:j\in S\}\cup \bigcup_{j=1}^{\ell} I_{2}(X_{C_{j}\setminus S}) \;\cup\; \bigcup_{m=1}^{r} \{\,(p_{m})_{T} : T\in \mathcal{T}(S)\,\},
\end{equation}
where $X_{C_{j}\setminus S}$ denotes the submatrix of $X$ whose columns are indexed by $C_{j}\setminus S$, and $I_{2}(X_{C_{j}\setminus S})$ denotes the set of its $2$-minors.
\end{proposition}

\begin{proof}
Let $\mathcal{F}$ denote the family of polynomials in~\eqref{ecu}, and set
\[\mathcal{F}_{1}=\{x_{i,j}:j\in S\}, \qquad
\mathcal{F}_{2} = \bigcup_{j=1}^{\ell} I_{2}(X_{C_{j}\setminus S}), 
\qquad 
\mathcal{F}_{3} = \bigcup_{m=1}^{r} \{\,(p_{m})_{T} : T\in \mathcal{T}(S)\,\}.
\]
We aim to show that $V(\mathcal{F}) = V_{S,i}$, where $V_{S,i} = \overline{U_{S,i}}$.

\medskip
\noindent
\textbf{Step 1:} ($V_{S,i}\subseteq V(\mathcal{F})$)
To prove this inclusion, it suffices to show that $\mathcal{F}\subseteq I(U_{S,i})$. Since 
$\gamma_{j}=0$ for all $j\in S$ and $\gamma\in U_{S}$, we conclude that $\mathcal{F}_{1}\subseteq I(U_{S,i})$.
By the definition of $V_{\Delta}$, we have
$\bigcup_{j=1}^{\ell} I_{2}(X_{C_{j}})\subseteq I(V_{\Delta}),$
which implies
$\mathcal{F}_{2} \subseteq \bigcup_{j=1}^{\ell} I_{2}(X_{C_{j}}) \subseteq I(V_{\Delta}) \subseteq I(U_{S,i}).$
Moreover, by Proposition~\ref{prop p}, each $(p_{m})_{T}$ belongs to $I(U_{S,i})$, hence $\mathcal{F}_{3}\subseteq I(U_{S,i})$ as well.  
Therefore, $\mathcal{F}\subseteq I(U_{S,i})$, which yields $V_{S,i}\subseteq V(\mathcal{F})$.

\medskip
\noindent
\textbf{Step 2:} ($V(\mathcal{F})\subseteq V_{S,i}$)
Let $\gamma\in V(\mathcal{F})$, and we show that $\gamma\in V_{S,i}$. Since $\gamma\in V(\mathcal{F}_{1})$, it follows that $\gamma_{j}=0$ for all $j\in S$. 
Since $\gamma\in V(\mathcal{F}_{2})$, for each $j\in [\ell]$ the set $\{\gamma_{u} : u\in C_{j}\setminus S\}$ has rank at most one.  
Hence, there exists an $S$-representative $T=\{t_{1},\ldots,t_{\ell}\}$ such that, for every $j\in [\ell]$, all vectors $\{\gamma_{u}:u\in C_{j}\setminus S\}$ are scalar multiples of $\gamma_{t_{j}}$.

Since $\gamma\in V(\mathcal{F}_{3})$, we have $(p_{m})_{T}(\gamma)=0$ for all $m\in [r]$, which gives
$p_{m}(\gamma_{T}) = (p_{m})_{T}(\gamma) = 0$.
Thus $\gamma_{T}\in \overline{F_{S,i}}$, because $\{p_{1},\ldots,p_{r}\}$ is a complete set of defining equations for $F_{S,i}$.  
Since $F_{S,i}$ is quasi-affine, $\gamma_{T}$ belongs to its Zariski closure, so it can be approximated arbitrarily closely by a collection of vectors in $F_{S,i}$.  
Each vector of $\gamma$ being a scalar multiple of some vector in $\gamma_{T}$, a small perturbation of $\gamma_{T}$ within $F_{S,i}$ induces a corresponding perturbation of $\gamma$ within $U_{S,i}$.  
Hence $\gamma\in \overline{U_{S,i}} = V_{S,i}$. This completes the proof.
\end{proof}

Using these results, we reach the following conclusion.

\begin{remark}
By Theorems~\ref{irreducible} and~\ref{corresp}, Corollary~\ref{dimensions} and~Proposition~\ref{equations}, the determination of the irreducible components of $V_{\Delta}$, together with their dimensions and defining equations, reduces to analyzing the corresponding properties of the varieties $F_{S}$, where $S$ ranges over the admissible subsets.
\end{remark}

\begin{example}
Let $k$, $\ell$, $t$, $d$, and $S$ be as in Example~\ref{example: S1,S2,S3}. If 
$\{p_{1},\dots,p_{r}\}$ is a complete set of defining equations for $F_{S,i}$, then a complete set of defining equations for $V_{S,i}$ is
\begin{align*}
& \{x_{i,j} : j \in \{1,4,8,11,15,18\}\} \;\cup\; I_{2}(X_{2,3}) \;\cup\; I_{2}(X_{5,6}) \;\cup\; I_{2}(X_{7,9}) \\
& \;\;\;\;\cup\; I_{2}(X_{10,12}) \;\cup\; I_{2}(X_{13,14}) \;\cup\; I_{2}(X_{16,17}) \;\cup\; 
\bigcup_{m=1}^{r} \{\, (p_{m})_{T} : T \in \mathcal{T}(S) \,\}.
\end{align*}
\end{example}

\subsection{Constructing polynomials in \texorpdfstring{$I(F_{S})$}{I(FS)}}

In this subsection, we construct polynomials in $I(F_S)$ using rank constraints arising from intersections of matroid flats. 
These constraints translate into vanishing conditions for suitable minors, yielding explicit equations on $F_S$. 
We recall the following classical property of matroids; see \cite{Oxley}.

\begin{lemma}\label{flats}
Let $M$ be a matroid, and let $F$ and $G$ be flats such that $F \not\subseteq G$. Then $F \cap G$ is also a flat, and it satisfies
$\rank(F \cap G) < \rank(F).$
\end{lemma}

We now introduce a few additional notions.

\begin{definition}
Let $\mathcal{F}$ be a collection of subsets of $[n]$. A subset $F \subset [n]$ is said to be an \emph{$r$-intersection} 
if there exist elements $F_{1}, \ldots, F_{r} \in \mathcal{F}$ such that:
\[\bigcap_{i=1}^{j} F_{i} \not\subseteq F_{j+1}\ \text{for every }j \leq r-1;\quad \text{and}\quad
F = \bigcap_{i=1}^{r} F_{i}.
\]
For each $r \geq 1$, we denote by $\mathcal{F}_{r} \subset 2^{[n]}$ the collection
$\mathcal{F}_{r} = \{ F \subset [n] : \text{$F$ is an $r$-intersection} \}.$
\end{definition}
When $\mathcal{F}$ is the collection of flats of a matroid of fixed rank, we obtain the following consequence.

\begin{lemma}\label{flats 2}
Let $M$ be a matroid on $[n]$ and let $m\leq n$ be a fixed positive integer. Let $\mathcal{F}$ denote the collection of flats of rank $m$ of $M$. Then, for any $F \in \mathcal{F}_{r}$, we have
$\rank(F) \leq m - r + 1.$
\end{lemma}

\begin{proof}
The claim follows by successive application of Lemma~\ref{flats}.
\end{proof}

We use the previous lemma to construct polynomials in $\CC[Y]$ that vanish on $F_S$.

\begin{definition}
Let $S \subset [k\ell]$ be admissible, and let $\mathcal{F}^{S}=\{S_{1},\ldots,S_{k}\}$ be the collection of subsets of $[\ell]$ as given in Definition~\ref{matr aso}.  
For $r \ge 1$, denote by $\mathcal{F}_{r}^{S}$ the collection of $r$-intersections of $\mathcal{F}^{S}$.
\end{definition}

\begin{example}\label{example: type 3}
Let $k=3$, $\ell=8$, $d=t=4$, and $S=\{1,4,8,11,15,18\}$. In this case,
\[
\mathcal{F}^{S} = \{S_{1}, S_{2}, S_{3}\}
= \big\{\{3,4,5,6,7,8\},\, \{1,2,5,6,7,8\},\, \{1,2,3,4,7,8\}\big\}.
\]
Since $S_{1} \cap S_{2} \cap S_{3} = \{7,8\}$ and $S_{1} \cap S_{2} \not\subseteq S_{3}$, $\{7,8\}$ is a 3-intersection, or equivalently $\{7,8\} \in \mathcal{F}_{3}^{S}$.
\end{example}

Observe that, by the definition of $F_{S}$, for any $\gamma = (\gamma_{j} : j \in [k\ell]) \in F_{S}$,
the collection $\mathcal{F}^{S} = \{S_{1}, \ldots, S_{\ell}\}$ is contained in the family of flats of rank $t-1$ of the matroid associated with $\gamma$.  
This observation motivates the following construction of polynomials in $I(F_{S})$.

\begin{proposition}\label{prop:polynomials for FS}
Let $Y_{A}$ denotes the submatrix of $Y$ with columns indexed by $A$, and $I_{t - r + 1}(Y_{A})$ denotes the set of all $(t - r + 1)$-minors of this submatrix. Then, for each $r \in [t]$, we have
\[
\{I_{t - r + 1}(Y_{A}) : A \in \mathcal{F}^{S}_{r}\} \subset I(F_{S}),
\]
\end{proposition}

\begin{proof}
It suffices to show that $F_{S} \subset V(I_{t - r + 1}(Y_{A}))$ for each $A \in \mathcal{F}^{S}_{r}$.
Let $\gamma = (\gamma_{j} : j \in [k\ell]) \in F_{S}$, and let $\mathcal{N}_{\gamma}$ be the matroid on $[\ell]$ determined by the vectors of $\gamma$.  
By Definition~\ref{fs}, each element of $\mathcal{F}^{S}$ is a flat of rank $t-1$ in $\mathcal{N}_{\gamma}$.  
Applying Lemma~\ref{flats 2}, we obtain
$\rank\{\gamma_{a} : a \in A\} \le t - r$,
for every $A \in \mathcal{F}^{S}_{r}$.  
Hence, all $(t - r + 1)$-minors of $Y_{A}$ vanish at $\gamma$, proving the claim.
\end{proof}

\begin{example}
Let $k, \ell, d, t$, and $S$ be as in Example~\ref{example: type 3}.  
Since $\{7,8\} \in \mathcal{F}_{3}^{S}$, Proposition~\ref{prop:polynomials for FS} implies  
$I_{2}(Y_{7,8}) \subset I(F_{S})$,
that is, all $2$-minors of the submatrix of $Y$ formed by columns $7$ and $8$ vanish on $F_{S}$.
\end{example}

We can also derive a necessary condition to ensure that $F_S$ is non-empty.

\begin{lemma}
If $\mathcal{F}_{t}^{S}\neq \emptyset$, then $F_{S}=\emptyset$.
\end{lemma}

\begin{proof}
Assume, for contradiction, that there exists $\gamma=(\gamma_{1},\ldots,\gamma_{\ell})\in F_{S}$.  
By definition, for each $i\in [k]$, the set $S_{i}$ is a flat of rank $t-1$ in the matroid $\mathcal{N}_{\gamma}$.  
Since $\mathcal{F}_{t}^{S}\neq \emptyset$, there is a nonempty subset $A\subset [n]$ such that $A\in \mathcal{F}_{t}^{S}$.  
Applying Lemma~\ref{flats 2}, we obtain $\operatorname{rank}(A)\le 0$.  
Hence every element of $A$ is a loop in $\mathcal{N}_{\gamma}$, which implies $\gamma_{a}=0$ for all $a\in A$, contradicting the assumption that the vectors of $\gamma$ are~nonzero.
\end{proof}

\section{Irreducible decomposition of \texorpdfstring{$V_{\Delta}$}{VDelta} for \texorpdfstring{$k=2$}{k=2}}\label{sec:VDelta-k=2}

In this section, we determine the irreducible components of $V_{\Delta}$ under the assumption $k=2$. We note that, for $k=2$, only the special case $d = t$ was previously resolved in~\cite{alexandr2025decomposing}; hence, the results presented here constitute a substantial generalization. Our first step is to establish the irreducibility of the varieties~$F_{S}^{j}$.

\subsection{Decomposing the varieties \texorpdfstring{$F_S$}{FS} for \texorpdfstring{$k=2$}{k=2}}\label{decompose_F_S}

In this subsection, we focus on the case $k=2$ and provide a decomposition of each $F_{S}$ for admissible subsets $S$ into a union of quasi-affine subvarieties $F_{S}^{i}$, defined as follows.

\begin{definition}
Define the subsets of $[\ell]$
\[
A \coloneqq \{ j \in [\ell] : \mathcal{Y}_{1,j} \notin S \}
\quad \text{and} \quad
B \coloneqq \{ j \in [\ell] : \mathcal{Y}_{2,j} \notin S \},
\]
where $\mathcal{Y}$ denotes the matrix from~\eqref{matrix}. Moreover, we set 
\[u \coloneqq  \ell - |A|=\size{S\cap R_{1}} \quad  \text{and} \quad v \coloneqq  \ell - |B|=\size{S\cap R_{2}},\]
and we say that $S$ is of {\em type} $(u,v)$.
For each $j$, we define the quasi-affine variety
\[
F_{S}^{j} \coloneqq 
\bigl\{ \gamma \in F_{S} : 
\dim\bigl( \Span \gamma_{A} \cap \Span \gamma_{B} \bigr) = j \bigr\}, \quad \text{where $\gamma_{A}=\{\gamma_{i}:i\in A\}$}.\]

\end{definition}

\begin{example}\label{example for k=2}
Let $\,\ell = 10$, $t = 5$, $d = 6$, and $S = \{1,3,5,16,18,20\}$.
In this case,
\[
A = \{4,5,6,7,8,9,10\}, 
\qquad 
B = \{1,2,3,4,5,6,7\},
\qquad 
u = v = 3.
\]

For instance, the quasi-affine variety $F_{S}^{2}$ consists of all tuples of vectors $\gamma = (\gamma_i : i \in [10])$
such that both $A$ and $B$ are flats of rank $4$ in $N_\gamma$, and
\[
\dim\!\left( 
\Span\{\gamma_i : 4 \leq i \leq 10\} 
\;\cap\; 
\Span\{\gamma_i : 1 \leq i \leq 7\}
\right) = 2.
\]
\end{example}

We then obtain a decomposition of $F_{S}$ as a union of the quasi-affine varieties $F_{S}^{i}$, which, a priori, need not be irreducible nor irredundant. 

\begin{theorem}\label{thm:decomposition}
     For any admissible subset $\emptyset \neq S\subseteq [2 \ell]$, we have 
        $F_S = \cup_{j \in \mathcal{P}(S)} F_{S}^j,$
    where 
    \begin{align*}
        &\mathcal{P}(S) = [x(S),t-2],\\
        & x(S) = \begin{cases}
            \max\{1, t-1-|A\setminus B|, t-1-|B\setminus A|, 2t-2-d\} & \text{ if } A \cap B \neq\emptyset, \\
            \max \{0, 2t-2-d\} & \text{ if }A \cap B = \emptyset.
        \end{cases}
    \end{align*}
\end{theorem}
\begin{proof}
    First note that for all $\gamma \in F_S$, we have
    \begin{align*}
        d \geq \rank(\gamma) &= \Dim(\Span \gamma_A) + \Dim(\Span \gamma_B) - \Dim(\Span \gamma_A \cap \Span \gamma_B)  \\
        &= \rank(A) + \rank(B) - \Dim(\Span \gamma_A \cap \Span \gamma_B) = 2t-2 - \Dim(\Span\gamma_A \cap \Span \gamma_B),
    \end{align*}
    and therefore, $\Dim(\Span \gamma_A \cap \Span \gamma_B) \geq 2t-2-d$. Moreover, if $A \cap B \neq \emptyset$, we have 
    \begin{align*}
        t-1=\rank(A)  = \rank((A\setminus B) \cup (A\cap B)) \leq \rank(A\setminus B) + \rank(A\cap B) \leq |A\setminus B| + \rank(A\cap B),
    \end{align*}
    and therefore,
    \begin{align*}
        \Dim(\Span \gamma_A \cap \Span \gamma_B) \geq \rank(\gamma_{A\cap B}) = \rank(A\cap B) \geq t-1-|A\setminus B|.
    \end{align*}
    Similarly, $ \Dim(\Span \gamma_A \cap \Span \gamma_B
    )\geq t-1 - |B\setminus A|$. Hence, for all $\gamma \in F_S$, we have $\Dim(\Span \gamma_A \cap \Span \gamma_B) \geq x(S)$. On the other hand, 
    \begin{align*}
        \Dim(\Span \gamma_A \cap \Span \gamma_B) \leq \Dim(\Span\gamma_A) -1 = t-2,
    \end{align*}
    since otherwise, $\Span\gamma_A = \Span\gamma_B$. But since $A$ and $B$ are both flats, this means $A = B$ or equivalently, $S=\emptyset$. Thus, for all $\gamma\in F_S$, $\Dim(\Span \gamma_A \cap \Span \gamma_B) \in \mathcal{P}(S)$. 
\end{proof}

\begin{example}
Let $\ell, t, d$, and $S$ be as in Example~\ref{example for k=2}.  
In this setting, we have $x(S) = 2$.  
Therefore, by Theorem~\ref{thm:decomposition} we obtain the decomposition
$F_{S} \;=\; F_{S}^{2} \,\cup\, F_{S}^{3}$.
\end{example}

\subsection{Ideals of the varieties \texorpdfstring{$F_{S}^j$}{Fsj}}\label{sec:Ideasl:F_S}

In this subsection, we describe a complete set of defining equations for the varieties $F_{S}^{j}$. 
These equations are given by minors of the $d \times \ell$ matrix of indeterminates $Y=(y_{i,j})$.

\begin{theorem}\label{thm:ideals}
For any admissible subsets $\emptyset \neq S \subseteq [2\ell]$ and any $j \in \mathcal{P}(S)$, we have $\overline{F_{S}^j} = V(I_{S,j})$, where
\[
I_{S,j} \coloneqq I_{t}(Y_{A}) + I_{t}(Y_{B}) + I_{j+1}(Y_{A \cap B}) + I_{2t - j - 1}(Y).
\]
\end{theorem}
\begin{proof}
First, note that all polynomials in $I_{S,j}$ must vanish on the matrices in $F_S^j$, by definition. Therefore, $\overline{F_S^j}\subseteq V(I_{S,j})$. 

To show the other containment, take $\gamma\in V(I_{S,j})$. For every $\varepsilon>0$, we prove there exists $\gamma'\in B_\varepsilon(\gamma)$ such that $\gamma'\in F_{S}^j$. Since the minors $I_{j+1}(Y_{A\cap B})$ vanish on $\gamma$, we have that $\rank(\gamma_{A\cap B}) = j_0$, where $j_0 \leq \min\{j,|A\cap B|\}$. Therefore, we can pick a basis $\mathcal{B}_2 \subseteq \mathbb{R}^d$ of size $j_0$ for $\Span \gamma_{A\cap B}$ such that $\gamma_{A\cap B} = (\mathcal{B}_2) N_2$ for some coefficient matrix $N_2 \in \mathbb{R}^{j_0 \times |A\cap B|}$. Now we define $a \coloneqq \rank(\gamma_A)$ and extend $\mathcal{B}_2$ to a basis $\mathcal{B}_1 \cup \mathcal{B}_2$ for $\gamma_A$ such that $|\mathcal{B}_1| = a - j_0$ and $\gamma_{A\setminus B} = (\mathcal{B}_1) N_{11} + (\mathcal{B}_2) N_{12}$ for some coefficient matrices $N_{11} \in \mathbb{R}^{(a-j_0) \times |A\setminus B|} $ and $N_{12} \in \mathbb{R}^{j_0 \times |A\setminus B|}$. Finally we extend $\mathcal{B}_1\cup \mathcal{B}_2$ to a basis $\mathcal{B}_1 \cup \mathcal{B}_2 \cup \mathcal{B}_3$ for $\gamma$. Define $b\coloneqq \rank(\gamma_B )$. Then for some $C\subseteq \mathcal{B}_1$ with $|C| = b-j_0-|\mathcal{B}_3|$, we have that $C \cup \mathcal{B}_2 \cup \mathcal{B}_3$ is a basis for $\Span (\gamma_B)$. 
We consider the following cases:

\textbf{1. Assume $\boldsymbol{0\leq b-j-|\mathcal{B}_3|}$.}
Since $j_0\leq j$, we can pick $C'\subseteq C$ such that $|C'| = b - j - |\mathcal{B}_3|$. Then there exist coefficient matrices $N_{31} \in \mathbb{R}^{(j-j_0) \times |B\setminus A|}$, $N_{32} \in \mathbb{R}^{j_0 \times |B\setminus A|}$, $N_{33} \in \mathbb{R}^{|\mathcal{B}_3|\times |B\setminus A|}$ and $N_{34} \in \mathbb{R}^{(b-j-|\mathcal{B}_3|)\times |B\setminus A|}$ such that 
\begin{align*}
    \gamma_{B\setminus A} = \begin{pmatrix} C\setminus C' & \mathcal{B}_2  & \mathcal{B}_3 & C' \end{pmatrix} \begin{pmatrix} N_{31} \\ N_{32} \\ N_{33}  \\ N_{34} \end{pmatrix}.
\end{align*}
Now by the construction above, $\gamma = M\cdot N$ where
\[
M = \left(\begin{array}{c !{\vrule width 1.5pt} ccc !{\vrule width 1.5pt} c !{\vrule width 1.5pt} cc !{\vrule width 1.5pt} c}
   0_{d\times(t-1-a)} & \mathcal{B}_1 \setminus C &  C' & C\setminus C' & \mathcal{B}_2 & \mathcal{B}_3 & C' & 0_{d\times(t-1-b)}
\end{array}\right), \quad N = \left(\begin{array}{c|c|c}
    0 & 0 & 0 
    \\ \noalign{\hrule height 1.5pt}
    N_{11} & 0 & \begin{array}{c}
           0\\ 
         0 \\
         N_{31}
    \end{array} \\
    \noalign{\hrule height 1.5pt}
    N_{12} & N_{2} & N_{32}\\
    \noalign{\hrule height 1.5pt}
    0 & 0 & \begin{array}{c} N_{33} \\ N_{34} \end{array} \\\noalign{\hrule height 1.5pt}
    0 & 0 & 0\\
\end{array}\right).
\]

Note that since $a = \rank(\gamma_A)\leq t-1$ and $b = \rank(\gamma_B)\leq t-1$, the matrix $M$ is well-defined. Define $\ell_0 \coloneqq \min\{j, |A\cap B|\}$. Now we can move up the second horizontal line in matrix $N$ so that the number of rows in the third block becomes equal to $\ell_0$, and then remove the first and fourth horizontal lines. Therefore, without loss of generality, $\gamma = M \cdot N$ such that 
\begin{align*}
    M \in \mathbb{R}^{d\times (2t-2-j)}, \ N = \left(\begin{array}{c|c|c}K_{11} & 0 & \begin{array}{c}  0 \\ K_{31} \end{array}\\\hline K_{12}  & K_2 & K_{32}  \\\hline 0 & 0 & K_{33}\end{array}\right) \begin{array}{c}
    \text{ with } K_{11} \in \mathbb{R}^{(t-1-\ell_0) \times |A\setminus B|}, K_{12} \in \mathbb{R}^{\ell_0 \times |A\setminus B|}, \\
    K_2 \in \mathbb{R}^{\ell_0 \times |A\cap B|},
    K_{31} \in \mathbb{R}^{(j - \ell_0) \times |B \setminus A|}, \\K_{32}\in \mathbb{R}^{\ell_0\times |B\setminus A|}, K_{33} \in \mathbb{R}^{(t-1-j )\times |B\setminus A|}.
    \end{array}
\end{align*}

\textbf{2. Assume $\boldsymbol{b-j-|\mathcal{B}_3| < 0\leq t-1-j-|\mathcal{B}_3|}$.} In this case, since $C \cup \mathcal{B}_2 \cup \mathcal{B}_1$ is a basis for $\gamma_{B}$, there exist coefficient matrices $N_{31} \in \mathbb{R}^{(b-j_0-|\mathcal{B}_3
|) \times |B\setminus A|}$, $N_{32} \in \mathbb{R}^{j_0 \times |B\setminus A|}$ and $N_{33} \in \mathbb{R}^{|\mathcal{B}_3| \times B\setminus A}$ such that 
\begin{align} \label{eq: gamma BA}
    \gamma_{B\setminus A} = \begin{pmatrix} C & \mathcal{B}_2 & \mathcal{B}_3 \end{pmatrix} \begin{pmatrix} N_{31} \\ N_{32} \\ N_{33} \end{pmatrix}. 
\end{align}
Therefore, $\gamma = M \cdot N$ such that
\begin{align*}
    M  = \left(\begin{array}{c !{\vrule width 1.5pt}  cc !{\vrule width 1.5pt} c !{\vrule width 1.5pt}  c  !{\vrule width 1.5pt} c}
    0_{d\times (t-1-a)} & \mathcal{B}_1\setminus C & C & \mathcal{B}_2  & \mathcal{B}_3 & 0_{d \times (t-1-j-|\mathcal{B}_3|)}
    \end{array}\right),
    N = \left(\begin{array}{c|c|c}
    0 & 0 & 0\\
    \noalign{\hrule height 1.5pt}
    N_{11} &  0 &  \begin{array}{c} 0 \\ N_{31} \end{array} \\ \noalign{\hrule height 1.5pt}
    N_{12} & N_2 & N_{32} \\ \noalign{\hrule height 1.5pt} 0 & 0 & N_{33} \\ \noalign{\hrule height 1.5pt} 0 & 0 & 0
    \end{array}
    \right).
\end{align*}
The matrix $M$ is well-defined since $a = \rank(\gamma_A)\leq t-1$. Again we move up the second horizontal line in $N$ to make the number of rows in the third block equal to $\ell_0$, and then remove the first and fourth horizontal lines. Then we get that $\gamma = M \cdot N$, where
\begin{align*}
    M \in \mathbb{R}^{d \times (2t-2-j)}, \ 
    N = \left(\begin{array}{c | c | c}  K_{11} & 0 & \begin{array}{c}  0 \\ K_{31} \end{array} \\\hline K_{12} & K_2 & K_{32} \\\hline 0 & 0 & K_{33}\end{array}\right)
    \begin{array}{c}
    \text{with } K_{11} \in \mathbb{R}^{(t-1-\ell_0) \times |A \setminus B|}, K_{12} \in \mathbb{R}^{\ell_0 \times |A\setminus B|}, \\
    K_2 \in \mathbb{R}^{\ell_0 \times |A\cap B|} , K_{31} \in \mathbb{R}^{(j-\ell_0) \times |B\setminus A|}, \\
    K_{32}\in \mathbb{R}^{\ell_0 \times |B\setminus A|}, K_{33}\in \mathbb{R}^{(t-1-j) \times |B\setminus A|}.
    \end{array}
\end{align*}

\textbf{3. Assume $\boldsymbol{ t-1-j-|\mathcal{B}_3|<0}$.} In this case, note that \eqref{eq: gamma BA} still holds. Thus, $\gamma = M \cdot N$ such that 
\begin{align*}
  M  = \left(\begin{array}{c !{\vrule width 1.5pt}  cc !{\vrule width 1.5pt} c !{\vrule width 1.5pt}  c }
    0_{d\times (2t-2-j-a-|\mathcal{B}_3|)} & \mathcal{B}_1\setminus C & C & \mathcal{B}_2  & \mathcal{B}_3
    \end{array}\right),
    N = \left(\begin{array}{c|c|c}
    0 & 0 & 0\\
    \noalign{\hrule height 1.5pt}
    N_{11} &  0 &  \begin{array}{c} 0 \\ N_{31} \end{array} \\ \noalign{\hrule height 1.5pt}
    N_{12} & N_2 & N_{32} \\ \noalign{\hrule height 1.5pt} 0 & 0 & N_{33} 
    \end{array}
    \right).    
\end{align*}
The matrix $M$ is well-defined since $a + |\mathcal{B}_3| = \rank(\gamma) \leq 2t-2-j$. Now we move up the second horizontal line in $N$ so that the third block in $N$ has $\ell_0$ rows. We then remove the first horizontal line. Therefore, we get that $\gamma = M \cdot N$, where
\begin{align*}
     M \in \mathbb{R}^{d\times (2t-2-j)}, \ N = \left(\begin{array}{c|c|c}K_{11} & 0 & \begin{array}{c}  0 \\ K_{31} \end{array}\\\hline K_{12}  & K_2 & K_{32}  \\\hline \begin{array}{c} K_{13} \\ 0  \end{array} & 0 & K_{33}\end{array}\right) 
     \hspace{-2.5mm}\begin{array}{c}
    \text{ with } \\ K_{11} \in \mathbb{R}^{(2t-2-j-\ell_0 - |\mathcal{B}_3|) \times |A\setminus B|}, K_{12} \in \mathbb{R}^{\ell_0 \times |A\setminus B|}, \\
    K_{13} \in \mathbb{R}^{(-t+1+j+|\mathcal{B}_3|)\times |B\setminus A|},
    K_2 \in \mathbb{R}^{\ell_0 \times |A\cap B|},\\
    K_{31} \in \mathbb{R}^{(t-1 - \ell_0 - |\mathcal{B}_3|) \times |B \setminus A|}, K_{32}\in \mathbb{R}^{\ell_0\times |B\setminus A|}, \\ K_{33} \in \mathbb{R}^{|\mathcal{B}_3|\times |B\setminus A|}.
    \end{array}
\end{align*}

 In all cases, perturb the matrices such that $M$ changes to a full column-rank matrix $M'$, and $\begin{pmatrix} K_{11} \\ K_{13} \end{pmatrix}$, $K_2$, and $\begin{pmatrix} K_{31} \\ K_{33}\end{pmatrix}$ change to full row-rank matrices $\begin{pmatrix} K_{11}' \\ K_{13}' \end{pmatrix}$, $K_2'$ and $K_{31}'$ respectively. This is possible since the number of columns of $M$ is larger than or equal to its number of rows, and the numbers of rows of the other matrices are larger than or equal to their numbers of columns. We can show that $\gamma' = M' \cdot N' \in F_S^j$.
 \end{proof}

\begin{example}\label{ex:6*10}
Consider $\ell, t, d$, and $S$ as in Example~\ref{example for k=2}, and let 
$Y = (y_{i,j})$ be the $6 \times 10$ matrix of variables.  
Applying Theorem~\ref{thm:ideals} in this setting yields
\begin{equation}\label{IS2}
I_{S,2}
=
I_{5}\bigl(Y_{4,5,6,7,8,9,10}\bigr)
+
I_{5}\bigl(Y_{1,2,3,4,5,6,7}\bigr)
+
I_{3}\bigl(Y_{4,5,6,7}\bigr),
\end{equation}
and
\begin{equation}\label{IS3}
I_{S,3}
=
I_{5}\bigl(Y_{4,5,6,7,8,9,10}\bigr)
+
I_{5}\bigl(Y_{1,2,3,4,5,6,7}\bigr)
+
I_4\bigl(Y_{4,5,6,7}\bigr) +
I_{6}(Y),
\end{equation}
where $Y_{A}$ denotes the submatrix of $Y$ with columns indexed by $A$. Out of the 210 $6$-minors that generate $I_{6}(Y)$, 79 of them are also minimal generators of $I_{S,3}$, as can be confirmed computationally in \texttt{Macaulay2}.
\end{example}

\subsection{Irreducibility of the varieties \texorpdfstring{$F_{S}^j$}{FSj}}

Here, we establish the irreducibility of the varieties $F_S^j$ . 
The proof proceeds by realizing $F_S^j$ as the image of a polynomial map from an irreducible variety, and showing that its Zariski closure coincides with $\overline{F_S^j}$.

\begin{theorem}\label{thm:irreducible}
 For every admissible subset $  S \subset [2\ell]$ and every $j \in \mathcal{P}(S)$, the 
 variety $F_{S}^{j}$ is irreducible. 
\end{theorem}
\begin{proof}
    For $S\neq \emptyset$, let $u \coloneqq  \ell - |A|$ and $v \coloneqq  \ell - |B|$. Define the map
    \begin{align*}
        &\varphi : \mathbb{C}^{d \times (2t-2-j)} \times \mathbb{C}^{(t-1) \times u} \times \mathbb{C}^{j  \times (\ell - u -v)} \times \mathbb{C}^{(t-1) \times v} \to \mathbb{C}^{d \times \ell} \\
        & (M, N_1, N_2, N_3) \mapsto 
        \begin{pmatrix}
            M[1,t-1]\cdot N_1 & M[t-j,t-1]\cdot N_2 & M[t-j, 2t-2-j] \cdot N_3
        \end{pmatrix}.
    \end{align*}
We will show that $\overline{\text{Im}(\varphi)} = \overline{F_{S}^j}$. Since the domain of $\varphi$ is an irreducible variety, this will prove that $\overline{F_{S}^j}$ is irreducible. 
Indeed, consider a matrix $\gamma = \varphi(M, N_1, N_2, N_3)$. Then, the first $\ell - v$ columns of $\gamma$ are spanned by $M[1, t-1]$, while the last $\ell - u$ columns are spanned by $M[t - j, 2t - 2 - j]$ by construction. This translates to $\rank(\gamma_A) \leq t-1$ and $\rank(\gamma_B) \leq t - 1$. Therefore, the ideals $I_t(Y_A)$ and $I_t(Y_B)$ must vanish on $\gamma$. Similarly, by construction $\rank (\gamma_{A\cap B}) \leq j$. Therefore, the ideal $I_{j+1}(Y_{A\cap B})$ vanishes on $\gamma$. Moreover, by construction, the columns of $\gamma$ are spanned by the columns of $M$, and since $M$ has $2t-2-j$ columns, the ideal $I_{2t-j-1}(Y)$ vanishes on $\gamma$. We conclude that $\gamma \in \overline{F_S^j}$ by Theorem \ref{thm:ideals}.

For the other direction, consider a  $d\times \ell$ matrix $\gamma \in F_S^j$. By definition of $F_S^j$, 
   $\Dim(\Span\gamma_A \cap \Span \gamma_B) = j.$
So, we can select a basis $\mathcal{B}_2$ of size $j$ for $\Span\gamma_A \cap \Span \gamma_B$. Note that $\Span{\mathcal{B}_2} \supseteq \Span \gamma_{A\cap B}$. On the other hand, $A$ is a flat of rank $t - 1$ in $\mathcal{N}_\gamma$. In particular, the vectors in $\gamma_A$ span a space of rank $t - 1$, which contains $\Span\mathcal{B}_2$. Therefore, we can extend $\mathcal{B}_2$ to a basis $\mathcal{B}_1\cup \mathcal{B}_2$ of $\Span\gamma_A$. Similarly, we can extend $\mathcal{B}_2$ to a basis $\mathcal{B}_2\cup \mathcal{B}_3$ of $\Span\gamma_B$. Therefore, letting 
\begin{align*}
        M \coloneqq \begin{pmatrix} \mathcal{B}_1  & \mathcal{B}_2 & \mathcal{B}_3 \end{pmatrix},
    \end{align*}
and selecting $N_1, N_2, N_3$ to contain the appropriate coefficients, we get  $\gamma = \varphi(M, N_1, N_2, N_3)$. Therefore, $F_S^j\subseteq \text{Im}(\varphi).$ Taking Zariski closure, we obtain the desired result.
\end{proof}

\begin{corollary}\label{cor: are prime}
    For every admissible subset $S\subseteq [k\ell]$ and every $j \in \mathcal{P}(S)$, the ideal $I_{S,j}$ is prime. 
\end{corollary}

\begin{proof}
    By~\cite{seccia2022knutson}, the canonical generating set of the ideal $I_{S,j}$ from Theorem~\ref{thm:ideals} is a Gröbner basis for this ideal. In particular, $I_{S,j}$ has a square-free Gr\"obner basis, and is therefore radical. On the other hand, by Theorem~\ref{thm:ideals}, $V(I_{S,j}) = \overline{F_S^j}$. Since by Theorem~\ref{thm:irreducible}, $\overline{F_S^j}$ is irreducible, we conclude that $I_{S,j}$ is prime.
\end{proof}

\begin{example}
In the setting of Example~\ref{ex:6*10}, the ideals $I_{S,2}$ and $I_{S,3}$ from~\eqref{IS2} and~\eqref{IS3} are prime.
\end{example}

\subsection{Identifying the maximal varieties \texorpdfstring{$F_{S}^{j}$}{FSj}}\label{sec:maximal}

Fix an admissible subset $S$ and consider a variety $F_{S}^{j}$ in the decomposition of $F_{S}$ from Theorem~\ref{thm:decomposition}, with $j\in\mathcal{P}(S)$. 
We call $F_S^j$ \emph{maximal} if there is no $i\in\mathcal{P}(S)$ with $i\neq j$ such that
$\overline{F_S^j}\subseteq \overline{F_S^i}$.

The following theorem characterizes the indices $j\in\mathcal P(S)$ for which $F_S^j$ is maximal.

\begin{theorem}\label{thm:maximal varieties FSj}
Let $\emptyset \neq S \subseteq [k\ell]$ be an admissible subset and $j \in \mathcal{P}(S)$.  
Then $F_{S}^{j}$ is maximal~if~and~only~if
\begin{equation}\label{maximal j}
j \in 
\begin{cases}
[x(S),\, |A \cap B|], & \text{if } x(S) \le |A \cap B| \le t-2,\\[3pt]
[x(S),\, t-2], & \text{if } |A \cap B| > t-2,\\[3pt]
\{x(S)\}, & \text{if } |A \cap B| < x(S).
\end{cases}
\end{equation}
\end{theorem}

\begin{proof}
    Let $j=x(S)$ or $x(S) < j\leq \min\{t-2,|A\cap B|\}$. We show that $\overline{F_S^j}$ is maximal. Suppose, toward a contradiction, that $\overline{F_S^j}\subseteq \overline{F_S^i}$ for some $i\in\mathcal P(S)$ with $i\neq j$.
On the one hand, for all $\gamma\in F_S^j$, we have 
    \begin{align*}
        \Dim(\Span \gamma) = \Dim(\Span \gamma_A) + \Dim(\Span \gamma_B) - \Dim(\Span \gamma_A \cap \Span \gamma_B) = 2t-2-j.
    \end{align*}
Since $\gamma \in \overline{F_S^i} = V(I_{S,i})$, we have $\Dim(\Span \gamma) \leq 2t-2-i$. Therefore, $2t-2-j\leq 2t-2-i$, which implies $i\leq j$. This already yields a contradiction in the case $j=x(S)$, proving that $F_S^j$ is maximal in this case.

    On the other hand, since $2t-2-j\leq d$, we can pick $2t-2-j$ linearly independent vectors $v_1, \ldots, v_{2t-2-j} \in \mathbb{R}^d$. Consider arbitrary full-rank matrices $N_1 \in \mathbb{R}^{(t-1-j) \times u}$, $N_2 \in  \mathbb{R}^{j  \times (\ell - u -v)}$ and $ N_3 \in \mathbb{R}^{(t-1-j) \times v} $ such that none of them has 0 columns. Define
    \begin{align*}
     \gamma \coloneqq  \begin{pmatrix}
        \begin{pmatrix} | & \cdots & | \\ v_1 & \cdots & v_{t-1-j} \\ | & \cdots & | \end{pmatrix} N_1 & \begin{pmatrix} | & \cdots & | \\ v_{t-j} & \cdots & v_{t-1} \\ | & \cdots & | \end{pmatrix} N_2 & \begin{pmatrix} | & \cdots & | \\ v_{t} & \cdots & v_{2t-2-j} \\ | & \cdots & | \end{pmatrix} N_3 
     \end{pmatrix} \in \mathbb{R}^{d\times \ell}.
    \end{align*}
    Since $j\geq x(S)$, we have $t-1-j \leq |A\setminus B| = v$ and $t-1-j\leq |B\setminus A| = u$. So, $\rank(N_1) = \rank(N_3) = t-1-j$. Hence,  $\Dim(\Span \gamma_{A\setminus B}) = \Dim(\Span \gamma_{B\setminus A}) = t-1-j$. Moreover, $j\leq |A\cap B| = \ell-u-v$, and so, $\rank(N_2) = j$. Thus, $\Dim(\Span \gamma_{A\cap B})=j$. Therefore, $\gamma\in F_S^j$ with $\Dim(\Span \gamma_A \cap \Span \gamma_B) = \Dim(\Span \gamma_{A\cap B})=j$.
    Now since $\gamma\in F_S^j$, we have $\gamma\in \overline{F_S^i}$, and hence,  $\gamma \in V(I_{i+1} (Y_{A\cap B}))$. So, $j=\Dim(\Span \gamma_{A \cap B})  \leq i$. Therefore, $j=i$, a contradiction, which proves the maximality of $\overline{F_S^j}$.

  We now show that for $j\in\mathcal P(S)$ with $j\neq x(S)$ and $j>|A\cap B|$, the variety $\overline{F_S^j}$ is not maximal.
This is done by showing that $I_{S,i} \subseteq I_{S,j}$ for some $j\neq i \in \mathcal{P}(S)$.
In particular, if $x(S)\leq |A\cap B|$, we take $i=|A\cap B|$, and if $|A\cap B|<x(S)$, we take $i=x(S)$.
\end{proof}
\begin{example}
Following~Example~\ref{example for k=2},~we have $x(S)=2$ and $\lvert A\cap B\rvert=4>3=t-2$. 
By Theorem~\ref{thm:maximal varieties FSj}, $F_S^j$ is maximal 
for $j\in\{2,3\}$, and 
$\overline{F_S}=\overline{F_S^{2}}\cup\overline{F_S^{3}}$ is the irredundant irreducible decomposition of~$\overline{F_S}$.
\end{example}

\subsection{Irreducible decomposition of \texorpdfstring{$V_{\Delta}$}{VDelta}}
In this subsection, we describe the irreducible decomposition of $V_{\Delta}$ for $k=2$. 
We begin by giving the corresponding 
irreducible decompositions of the varieties $F_S$ and $V_S$ for admissible subsets $S \subset [2\ell]$.

\begin{proposition}\label{prop:irr dec FS}
Let $\emptyset \neq S \subset [2\ell]$ be an admissible subset.  The variety $\overline{F_{S}}$ admits the irredundant irreducible decomposition
\(\overline{F_{S}} = \bigcup_{j} \overline{F_{S}^{j}}\),
where $j$ ranges over all values specified in~\eqref{maximal j}.
\end{proposition}

\begin{proof}
By Theorem~\ref{thm:decomposition}, the variety $F_{S}$ decomposes as the union of the subvarieties $F_{S}^{j}$ for $j \in \mathcal{P}(S)$.  
Theorem~\ref{thm:maximal varieties FSj} identifies the maximal varieties in this decomposition as those indexed by the values of $j$ given in~\eqref{maximal j}, giving both the equality and the irredundancy of the decomposition.  
Finally, the irreducibility of each component follows from Theorem~\ref{thm:irreducible}.
\end{proof}

\begin{corollary}
Let $\emptyset \neq S \subset [2\ell]$ be an admissible subset.  
The primary decomposition of the ideal~$\mathbb{I}(F_{S})$~is
\[
\mathbb{I}(F_{S}) = 
\bigcap_{j} 
\bigl( I_{t}(Y_{A}) + I_{t}(Y_{B}) + I_{j+1}(Y_{A \cap B}) + I_{2t - j - 1}(Y_{A \cup B}) \bigr),
\]
where $j$ ranges over all values specified in~\eqref{maximal j}.
\end{corollary}

\begin{proof}
The statement follows directly from Proposition~\ref{prop:irr dec FS}, Theorem~\ref{thm:ideals}, and Corollary~\ref{cor: are prime}.
\end{proof}

\begin{example}
Consider $t,\ell,d$, and $S$ as in Example~\ref{example for k=2}.  
The primary decomposition of $\mathbb{I}(F_{S})$ is
\begin{small}
\begin{align*}
\mathbb{I}(F_{S})
&=
\Bigl(
 I_{5}\bigl(Y_{4,5,6,7,8,9,10}\bigr)
 + I_{5}\bigl(Y_{1,2,3,4,5,6,7}\bigr)
 + I_{3}\bigl(Y_{4,5,6,7}\bigr)
\Bigr)
\\[4pt]
&\qquad\cap\;
\Bigl(
 I_{5}\bigl(Y_{4,5,6,7,8,9,10}\bigr)
 + I_{5}\bigl(Y_{1,2,3,4,5,6,7}\bigr)
 + (\det(Y_{4,5,6,7}))
 + I_{6}(Y)
\Bigr).
\end{align*}
\end{small}
\end{example}

We now focus on determining the irreducible components of $V_{S}$.

\begin{definition}
Let $\emptyset \neq S \subset [2\ell]$ be an admissible subset.  
By the one-to-one correspondence between the irreducible components of $F_{S}$ and those of $U_{S}$ established in Theorem~\ref{corresp}, we define
\[
U_{S}^{j} \coloneqq \psi\bigl( (\mathbb{C}^{\ast})^{\lvert [k\ell] \setminus S \rvert} \times F_{S}^{j} \bigr)
\]
to be the irreducible component of $U_{S}$ corresponding to $F_{S}^{j}$, where $\psi$ denotes the map introduced in~\eqref{psi}.  
We further set $V_{S}^{j}$ to be the Zariski closure of $U_{S}^{j}$.
\end{definition}

We now give defining equations for the varieties $V_S^j$ in terms of minors of the $d\times 2\ell$ matrix of indeterminates $X=(x_{i,j})$.
For a subset $D\subset[\ell]$, we let
$D' \coloneqq \bigcup_{j\in D}\{2j-1,2j\}.$

\begin{proposition}\label{thm:ideals 2}
Let $\emptyset \neq S \subseteq [2\ell]$ be an admissible subset and  $j \in \mathcal{P}(S)$. Then
$V_{S}^{j} = V(J_{S,j})$,
where
\begin{equation}\label{ideal JSj}
J_{S,j} \coloneqq (x_{i,j} : j \in S)
    + \sum_{j = 1}^{\ell} I_{2}(X_{C_{j}})
    + I_{t}(X_{A'}) + I_{t}(X_{B'})
    + I_{j + 1}(X_{A' \cap B'}) + I_{2t - j - 1}(X).
\end{equation}
Moreover, the ideal $J_{S,j}$ is prime. 
\end{proposition}

\begin{proof}
By Theorem~\ref{thm:ideals}, the ideal $I_{S,j}$ provides a complete set of defining equations for $F_{S}^{j}$.  
Furthermore, Proposition~\ref{equations} describes how to obtain defining equations for $V_{S}^{j}$ from those of $F_{S}^{j}$.  
Applying this result in the present context yields the ideal $J_{S,j}$, and hence $V_{S}^{j} = V(J_{S,j})$. The primeness of $J_{S,j}$ follows from applying exactly the same argument as in Corollary~\ref{cor: are prime}.
\end{proof}

\begin{proposition}\label{prop:irr dec VS}
Let $\emptyset\neq S\subset[2\ell]$ be admissible, and let $j$ range over the values in~\eqref{maximal j}.
 Then
\[
V_S=\bigcup_j V_S^j
\quad\text{and}\quad
\mathbb{I}(V_S)=\bigcap_j J_{S,j}
\]
give an irredundant irreducible decomposition of $V_S$ and the primary decomposition of $\mathbb{I}(V_S)$, respectively.
\end{proposition}

\begin{proof}
The irreducible decomposition follows from Proposition~\ref{prop:irr dec FS} together with the correspondence between the irreducible components of $F_{S}$ and those of $U_{S}$ established in Theorem~\ref{corresp}.  
The description of the primary decomposition of $\mathbb{I}(V_{S})$ then follows from Theorem~\ref{thm:ideals 2}.
\end{proof}

\begin{example}
Consider $t,\ell,d$ and $S$ as in Example~\ref{example for k=2}.  
Then, the primary decomposition of $\mathbb{I}(V_{S})$ is
\begin{small}
\begin{align*}
\mathbb{I}(V_{S})
&=
\Bigl(
 (x_{i,j} : j \in \{1,3,5,16,18,20\})
 \;+\;
 I_{2}(X_{7,8}) + I_{2}(X_{9,10}) + I_{2}(X_{11,12}) + I_{2}(X_{13,14})
\\[-2pt]
&\hspace{4.2cm}
 +\; I_{5}(X_{7,\ldots,20})
 +\; I_{5}(X_{1,\ldots,14})
 +\; I_{3}(X_{7,\ldots,14})
\Bigr)
\\[6pt]
&\qquad \cap\;
\Bigl(
 (x_{i,j} : j \in \{1,3,5,16,18,20\})
 \;+\;
 I_{2}(X_{7,8}) + I_{2}(X_{9,10}) + I_{2}(X_{11,12}) + I_{2}(X_{13,14})
\\[-2pt]
&\hspace{4.2cm}
 +\; I_{5}(X_{7,\ldots,20})
 +\; I_{5}(X_{1,\ldots,14})
 +\; I_{4}(X_{7,\ldots,14})
 +\; I_{6}(X)
\Bigr).
\end{align*}
\end{small}
\end{example}

We now arrive at the main result of this section, where we establish the irredundant irreducible decomposition of $V_{\Delta}$ under the assumption $k=2$.

\begin{theorem}\label{thm:decomposition of Vdelta for k=2}
For $k=2$, the variety $V_{\Delta}$ admits the following irredundant irreducible decomposition:
\[
V_{\Delta} = V_{\emptyset} \,\cup\, \bigcup_{S} \, \bigcup_{j} V(J_{S,j}),
\]
where the outer union runs over all admissible subsets $\emptyset \neq S \subset [2\ell]$, and for each such $S$, the second union runs over all values of $j$ specified in~\eqref{maximal j}. The ideals $J_{S,j}$ are as defined in~\eqref{ideal JSj}.
\end{theorem}

\begin{proof}
The result follows directly from Theorem~\ref{irreducible} and Proposition~\ref{prop:irr dec VS}.
\end{proof}

\begin{example}
Let $t=4$ and $d=\ell=5$.  
For $S\subset[10]$, we say that $S$ has \emph{type} $(u,v)$ if
$(|S\cap R_1|,|S\cap R_2|)=(u,v)$.
Then $S$ is admissible iff $u,v\le2$ and no column $C_j$ is contained in $S$.
Hence the admissible types are
\[
(0,0),\ (0,1),\ (0,2),\ (1,0),\ (2,0),\ (1,1),\ (1,2),\ (2,1),\ (2,2).
\]
Table~\ref{tab:maxj} lists the maximal values of $j$ for each admissible type $(u,v)$, as determined by~\eqref{maximal j}.

\begin{table}[H]
\centering
\begin{tabular}{c|c|c|c|c|c|c|c|c}
\text{type $(u,v)$}  & $(0,1)$ & $(0,2)$ & $(1,0)$ & $(2,0)$  & $(1,1)$ & $(1,2)$ & $(2,1)$ & $(2,2)$ \\
\hline
\text{maximal values of $j$}& $\emptyset$ & $\emptyset$ & $\emptyset$ & $\emptyset$ & $\{2\}$ & $\{2\}$ & $\{2\}$ & $\{1\}$
\end{tabular}
\caption{Values of $j$ for which $F_{S}^{j}$ is maximal for an admissible subset $S$ of type $(u,v)$.}
\label{tab:maxj}
\end{table}

From the table we see that $F_{S} = \emptyset$ for types  $(0,1), (0,2), (1,0), (2,0)$,
while $F_{S}$ is nonempty and irreducible for types  $(1,1), (1,2), (2,1), (2,2)$.  
By Theorem~\ref{thm:decomposition of Vdelta for k=2}, the irredundant irreducible decomposition of $V_{\Delta}$ is
\[
V_{\Delta} \;=\; 
V_{\emptyset}\ \cup\!
\bigcup_{\substack{
S \text{ admissible} \\
\text{of type } (1,1), (1,2), (2,1), (2,2)
}}
V_{S}.
\]
Finally, note that a fixed type $(u,v)$ has $\binom{5}{u}\binom{5-u}{v}$ admissible subsets.
Hence there are $20,30,30,$ and $30$ subsets of types $(1,1),(1,2),(2,1)$, and $(2,2)$, respectively, and $V_\Delta$ has exactly $111$ irreducible components.
\end{example}

\section{Dimension of the varieties \texorpdfstring{$F_{S}^j$}{FSj}, \texorpdfstring{$V_{S}^j$}{VSj}, and \texorpdfstring{$V_\Delta$}{VDelta} for the case \texorpdfstring{$k=2$}{k=2}}\label{sec:dim}
In this section, we determine the dimension of the variety $V_{\Delta}$ and of its irreducible components in the case $k=2$. 
Using the correspondence established in Notation~\ref{not corres}, we relate the dimensions of the components of $V_{\Delta}$ to those of the auxiliary varieties $F_S$. 
We begin by computing the dimensions of the varieties $F_S^j$, and then transfer these results to obtain the dimensions of the corresponding components $V_S^j$ and of $V_{\Delta}$.

\begin{theorem}\label{thm:dimension}
    For any non-empty admissible set $S$ and any $j\in \mathcal{P}(S)$, the dimension of $\overline{F_{S}^j}$ is 
\[
\mathcal{D} \coloneqq d(2t - 2 - j) + (t-1)(u+v)+ j(\ell-u-v) - (j^2 + 2(t-1-j)^2 + 2j(t-1-j)).
\]
\end{theorem}
\begin{proof}
    First, we will show that the dimension of $F_{S}^j$ is bounded above by $\mathcal{D}$ by showing that there is an action of $GL_j \times (GL_{t-1-j})^2 \times (\CC^{j\times (t-1-j)})^2$ on the parameter space that leaves the image unchanged. We will show that for any choice of parameters $(M, N_1, N_2, N_3)$, we have
    \begin{equation}\label{eq:image-unchanged}
    \varphi(M, N_1, N_2, N_3) = \varphi\left(MC, \begin{pmatrix}
        R_1 & 0 \\ Q_1 & P
    \end{pmatrix}^{-1}N_1,\quad P^{-1}N_2, \ \begin{pmatrix}
        P & Q_2 \\ 0 & R_2
    \end{pmatrix}^{-1}N_3\right),
     \end{equation}
where for some $P\in GL_j, R_1,R_2 \in GL_{t-1-j}, Q_1,Q_2 \in \mathbb{C}^{j \times (t-1-j)}$,
    \begin{align*}
        C = \begin{pmatrix}
            R_1 & 0_{(t-1-j) \times j} & 0_{(t-1-j) \times (t-1-j)} \\
            Q_1 & P & Q_2 \\
            0_{(t-1-j) \times (t-1-j)} & 0_{(t-1-j) \times j} & R_2
        \end{pmatrix}.
    \end{align*}

Write the matrix $M$ as $M = \begin{pmatrix}
    M_1 & M_2 & M_3
\end{pmatrix}$ where $M_1, M_3\in\CC^{d \times (t - 1 - j)}$ and $M_2 \in\CC^{d \times j}$. Then observe that $(MC)[t - j, t - 1] = M_2 P$ and
\begin{align*}
(MC)[1, t-1] = \begin{pmatrix}
    M_1 & M_2
\end{pmatrix}\begin{pmatrix}
    R_1 & 0 \\ Q_1 & P
\end{pmatrix} \text{ and } (MC)[t - j, 2t - 2 - j] = \begin{pmatrix}
    M_2 & M_3
\end{pmatrix}\begin{pmatrix}
    P & Q_2 \\ 0 & R_2
    \end{pmatrix}.
\end{align*}
Hence, \eqref{eq:image-unchanged} follows, and the generic fiber of $\varphi$ has dimension at least $j^2+2(t-1-j)^2+2j(t-1-j)$.

We now show that this upper bound is attained by proving that no larger image-preserving group acts on the parameters.
To see this, assume $\varphi(M,N_1,N_2,N_3) = \varphi(M', N_1', N_2', N_3')$ for two choices of parameters giving rise to a generic point in the image of $\varphi$. Then since $M$ and $M'$ are full-rank, there exists a unique matrix $C \in GL_{2t-2-j}$ such that $MC = M'$. Assume
\begin{align*}
    C = \begin{pmatrix}
            R_1 & T_1 & T_2\\
            Q_1 & P & Q_2 \\
            T_3 & T_4 & R_2
        \end{pmatrix},
\end{align*}
 with $P\in \mathbb{C}^{j\times j}$, $R_1,R_2 \in \mathbb{C}^{(t-1-j)\times (t-1-j)}$, $Q_1,Q_2 \in \mathbb{C}^{j \times (t-1-j)}$, $T_1,T_4 \in \mathbb{C}^{(t-1-j) \times j}$, $T_2,T_3 \in GL_{t-1-j}$. Once again, observe that we may write $M = \begin{pmatrix}M_1 & M_2 & M_3\end{pmatrix}$ and $M' = \begin{pmatrix}M_1' & M_2' & M_3'\end{pmatrix}$ for $M_1,M_1',M_3,M_3'\in \mathbb{C}^{d\times (t-1-j)}$ and $M_2,M_2'\in \mathbb{C}^{d\times j}$, so
\begin{align}\label{eq:M'-in-terms-of-M}
M' = MC = \begin{pmatrix}M_1R_1 + M_2Q_1 + M_3T_3 & M_1T_1 + M_2P + M_3T_4 & M_1T_2 + M_2Q_2 + M_3R_2
\end{pmatrix}.
\end{align}
Now without loss of generality, for some $W_1, W_1'\in \text{GL}_{t-1}$ and $W_2, W_2'\in\CC^{(d-t+1)\times(t-1)}$, re-write 
\[
\begin{pmatrix} M_1 & M_2 \end{pmatrix} = \begin{pmatrix} W_1 \\ W_2 \end{pmatrix} \text{ and } \begin{pmatrix} M_1' & M_2' \end{pmatrix} = \begin{pmatrix} W_1' \\ W_2' \end{pmatrix}.
\]
Then there exists a unique matrix $G\in {\rm GL}_{t-1}$ such that $W_1G=W_1'$. 
Now, define
\begin{align*}
    N \coloneqq \begin{pmatrix} N_1 & \begin{array}{c} 0_{(t-1-j) \times (\ell-u-v)} \\ N_2 \end{array} \end{pmatrix}, \quad
    N' \coloneqq \begin{pmatrix} N_1' & \begin{array}{c} 0_{(t-1-j) \times (\ell-u-v)} \\ N_2' \end{array} \end{pmatrix}.
\end{align*}
Since $\begin{pmatrix} M_1 & M_2 \end{pmatrix} N = \begin{pmatrix} M_1' & M_2' \end{pmatrix} N'$, we have $W_1N = W_1'N'$ and $W_2N = W_2'N'$. From the first equality we get $N = G N'$. From the second equality, $W_2 G N' =W_2' N'$. Since $S$ is admissible, we have $t-1 \leq \ell-v$, and so the matrix $N'$ has a right inverse. This means that $W_2 G  = W_2'$, therefore 
$\begin{pmatrix} M_1 & M_2 \end{pmatrix} G = \begin{pmatrix} M_1' & M_2' \end{pmatrix}$.
Combining this with \eqref{eq:M'-in-terms-of-M}, we get 
\begin{align}\label{eq:final-equation}
\begin{pmatrix} M_1 & M_2 \end{pmatrix} G = \begin{pmatrix}M_1R_1 + M_2Q_1 + M_3T_3 & M_1T_1 + M_2P + M_3T_4\end{pmatrix}.
\end{align}
Since every column on the left-hand side of \eqref{eq:final-equation} is a linear combination of $M_1$ and $M_2$, while every column on the right-hand side of \eqref{eq:final-equation} is a linear combination of $M_1$, $M_2$ and $M_3$, it follows that $T_3 = 0$ and $T_4 = 0$. Then note that by \eqref{eq:final-equation}, 
\begin{align*}
    \begin{pmatrix} M_1 & M_2 \end{pmatrix} G = \begin{pmatrix} M_1 & M_2 \end{pmatrix} \begin{pmatrix} R_1 & 0 \\ Q_1 & P \end{pmatrix}.
\end{align*}
Now since $d\geq t-1$, the matrix $\begin{pmatrix} M_1 & M_2 \end{pmatrix}$ has a left inverse, and therefore, 
    $G = \begin{pmatrix} R_1 & 0 \\ Q_1 & P \end{pmatrix},$
which forces $R_1\in\text{GL}_{t-1-j}$ and $P\in\text{GL}_{j}$.  
It can be proved similarly that $T_1 = 0$ and $T_2 = 0$ by considering the matrices $\begin{pmatrix}
    M_2 & M_3
\end{pmatrix}$ and $\begin{pmatrix}
    M_2' & M_3'
\end{pmatrix}$. Analogously, $R_2 \in \text{GL}_{t-1-j}$, and the conclusion follows.
\end{proof}

\begin{corollary} \label{cor: dimension of V_S^j}
For any admissible set $\emptyset \neq S\subset [k\ell]$ and any $j\in \mathcal{P}(S)$, the dimension of $V_{S}^{j}$ is 
\[d(2t - 2 - j) + (t-2)(u+v)+ j(\ell-u-v) - (j^2 + 2(t-1-j)^2 + 2j(t-1-j)) + \ell.\]
\end{corollary}

\begin{proof}
This follows from Theorem~\ref{thm:dimension} and Corollary~\ref{dimensions}, giving
$\Dim V_S^j=\Dim F_S^j+\ell-|S|$.
\end{proof}

We now identify the varieties $V_{S}^{j}$ that attain maximal dimension.

\begin{lemma}\label{top dimensional}
Among all varieties $V_{S}^{j}$, where $S \subset [kl]$ is a nonempty admissible subset and $j \in \mathcal{P}(S)$, the top-dimensional ones occur for $S_0$ with $u+v = \min\{2(\ell-t+1), \ell\}$ and $j_0 = x(S_0)$. In this case,
\begin{align}\label{eq: top dimensions}
\Dim V_{S_0}^{j_0} = 
\begin{cases}
  d\ell + 2t - \ell - 2 & \ell < 2t-2, \\
  (t - 1)(2d - 2t + \ell + 2) & \ell \geq 2t-2.
\end{cases}
\end{align}
\end{lemma}
\begin{proof}
For fixed parameters $j, t, \ell, d$, we have 
$\Dim(V_S^j) = (u + v)(t - 2 - j) + C(j, t, \ell, d)$
for some function $C$. Since $j \le t - 2$, the dimension is maximized when $u + v$ is maximized, i.e., $u + v = \min\{2(\ell - t + 1), \ell\}$. 
Moreover, for any nonempty admissible $S$, one has $\mathcal{P}(S) \subseteq \mathcal{P}(S_0)$, where $S_0$ is any set satisfying $u + v = \min\{2(\ell - t + 1), \ell\}$. 
A direct computation then gives 
\[
\Dim(V_{S_0}^j) = j(2t - 2 - d - j + \ell - \min\{2(\ell - t + 1), \ell\}) + C'(t, \ell, d)
\]
for some function $C'$. 
This is a quadratic polynomial in $j$, and by taking its derivative with respect to $j$, we get the critical point 
\[
    j = \frac12 \bigg(\underbrace{(2t - 2 - d)}_{\leq x(S_0)} + \underbrace{(\ell - \min\{2(\ell-t+1), \ell\})}_{\leq x(S_0)}\bigg), 
\]
which is less than or equal to $x(S_0)$
by the definition of $x(S_0)$ in Theorem~\ref{thm:decomposition}. 
Hence, the varieties of maximal dimension are precisely the $V_{S_0}^{j_0}$.
A case-by-case analysis yields the values of $u+v$ and $j_0$.
\[
\begin{array}{c@{\qquad}c}
\displaystyle
\min\{2(\ell-t+1), \ell\} =
\begin{cases}
2\ell - 2t + 2 & \ell < 2t-2 \le d, \\
2\ell - 2t + 2 & \ell \le d < 2t-2, \\
\ell           & d < 2t-2 \le \ell, \\
2\ell - 2t + 2 & d \le \ell < 2t-2, \\
\ell           & 2t-2 \le d \le \ell, \\
\ell           & 2t-2 \le \ell \le d,
\end{cases}
\quad\text{and}\quad
\displaystyle
j_0 =
\begin{cases}
2t-2-\ell & \ell < 2t-2 \le d, \\
2t-2-\ell & \ell \le d < 2t-2, \\
2t-2-d    & d < 2t-2 \le \ell, \\
2t-2-d    & d \le \ell < 2t-2, \\
0         & 2t-2 \le d \le \ell, \\
0         & 2t-2 \le \ell \le d.
\end{cases}
\end{array}
\]
Substituting these values into the formula for $\dim V_{S_0}^{j_0}$ from Corollary~\ref{cor: dimension of V_S^j} yields~\eqref{eq: top dimensions}.
\end{proof}


\begin{theorem} \label{thm:maximum dimension general}
The dimension of $V_{\Delta}$ is 
\begin{align*}
\begin{cases}
  \max\{d\ell + 2t - \ell - 2,\ (t - 1)(d + \ell - t + 1) + \ell\}& \ell < 2t-2, \\
  \max\{(t - 1)(2d - 2t + \ell + 2),\ (t - 1)(d + \ell - t + 1) + \ell\} & \ell \geq 2t-2.
\end{cases}
\end{align*}
\end{theorem}
\begin{proof}
The dimension of $V_\Delta$ equals the maximum of the dimensions of its irreducible components.
By Lemma~\ref{top dimensional}, the top-dimensional components have dimension
$(t-1)(2d-2t+\ell+2)$ if $\ell\geq 2t-2$, and $d\ell+2t-\ell-2$ otherwise.
Moreover, $\Dim V_{\emptyset}=(t-1)(d+\ell-t+1)+\ell$ by~\cite[Theorem~5.1]{alexandr2025decomposing}.
\end{proof}

\begin{example}
Consider $t,l,d$ as in Example~\ref{example for k=2}. In this setting we have $l=10\geq 8=2t-2$. Hence, by Theorem~\ref{thm:maximum dimension general}, we have
$\Dim(V_{\Delta})=\max\{56,58\}=58.$
\end{example}

\begin{remark}
    Note that the sign of the difference
\begin{align} \label{eq:difference}
   \dim V_{S_0}^{j_0} - \dim V_\emptyset = \begin{cases}
         (d-t)(\ell-t) -\ell+d-1  & \ell \leq 2t-2,\\
       (t-1)(d-t+1) - \ell & \ell \geq 2t -2.
   \end{cases} 
\end{align}
determines which of the two component types is maximal-dimensional. 
\end{remark}

These results will be used in the next section to compute the degree of $V_\Delta$.

\section{Degree of \texorpdfstring{$V_{\Delta}$}{VDelta} 
via lattice path enumeration}\label{sec:degree}

In this section, we compute the degree of the variety $V_{\Delta}$ associated with the ideal $I_{\Delta}$ from Definition~\ref{setup} in the case $k=2$. 
Using the dimension analysis from Section~\ref{sec:dim}, we identify the irreducible components of $V_{\Delta}$ contributing to its degree and reduce the computation of $\deg(V_{\Delta})$ to that of the component $V_{\emptyset}$.
The key observation, developed in \S\ref{sec: combinatorial approach to compute degree} and~\S\ref{sub:lattice_path}, is that $\deg(V_{\emptyset})$ admits a purely combinatorial description and can be expressed as a weighted count of families of non-intersecting lattice paths in a planar grid.
This reformulates the problem as a lattice path enumeration problem of Lindström--Gessel--Viennot type.
We develop this correspondence and use it to derive explicit degree formulas and generating functions.

\subsection{Reduction to the computation of the degree of \texorpdfstring{$V_{\emptyset}$}{Vempty}}
We begin by recalling the necessary results used in the proof, in order to keep the paper self-contained.

\begin{proposition}\label{prop:useful results}
The following statements hold:
\begin{itemize}
\item[{\rm (i)}] 
Let $I \subset \mathbb{C}[x_{1}, \ldots, x_{n}]$ be an ideal. Then 
$\deg(I) = \sum_{\mathfrak{p}} \deg(\mathfrak{p}),$
where the sum runs over all minimal primes $\mathfrak{p}$ of $I$ of maximal dimension. 

\item[{\rm (ii)}] 
Let $I \subset \mathbb{C}[x_{1}, \ldots, x_{n}]$ and $J \subset \mathbb{C}[y_{1}, \ldots, y_{m}]$ be ideals, and denote by $I'$ and $J'$ the ideals they generate in $\mathbb{C}[x_{1}, \ldots, x_{n}, y_{1}, \ldots, y_{m}]$. Then 
$\deg(I' + J') = \deg(I) \cdot \deg(J).$

\item[{\rm (iii)}] \textup{\cite[Theorem~3.5]{paths}} 
Let $Y = (y_{i,j})_{d \times \ell}$ be a $d \times \ell$ matrix of variables, and let $I_{t}(Y)$ be the ideal generated by the $t \times t$ minors of $Y$. Then 
\[
\deg(I_{t}(Y)) = 
\det\!\bigl[\tbinom{d + \ell - i - j}{\,d - i\,}\bigr]_{i,j = 1, \ldots, t-1.}
\]
\end{itemize}
\end{proposition}

\begin{proof}
  To prove (i), we use that for a prime decomposition of an ideal $I$, only the highest-dimensional components contribute to $\deg(I)$ \cite[Proposition~7.6]{hartshorne}. 
Indeed, the degree is the leading coefficient of the Hilbert polynomial, which is additive on exact sequences. 
Thus, it suffices to compute the degrees of the top-dimensional prime components in the decomposition of Theorem~\ref{thm:decomposition} and sum them.

To prove (ii), let $d_1 = \dim I'$ and $d_2 = \dim J'$. Since $I'$ and $J'$ involve disjoint sets of variables, $\Dim(I' + J') = d_1 + d_2$, and their Hilbert series \cite[Chapter 1]{MichalekSturmfels2019InvitationNonlinearAlgebraTEXT} multiply:
\[
HS_{I' + J'}(z) = HS_{I'}(z)HS_{J'}(z)
  = \frac{P_{I'}(z)}{(1-z)^{d_1}} \cdot \frac{P_{J'}(z)}{(1-z)^{d_2}}
  = \frac{P_{I'}(z)P_{J'}(z)}{(1-z)^{d_1 + d_2}}.
\]
Thus $P_{I'+J'}(z)=P_{I'}(z)P_{J'}(z)$.
For any standard graded $\CC$-algebra $R$ of dimension $d$, write
$HS_R(z)=P_R(z)/(1-z)^d$, where $P_R(z)$ is a polynomial with integer coefficients.
Since $(1-z)^{-d}=\sum_{q\ge0}\binom{q+d-1}{d-1}z^q$, it follows that for large enough $q\gg0$,
the Hilbert function of $R$ satisfies 
\[
h_R(q) = \frac{P_R(1)}{(d-1)!}\, q^{d-1} + \text{lower order terms}.
\]
By definition, the leading coefficient of the Hilbert polynomial gives the degree of $R$ (up to the factor $(d-1)!$), so $\deg(R) = P_R(1).$ Hence
$\deg(I' + J') = P_{I' + J'}(1) = P_{I'}(1)P_{J'}(1) = \deg(I')\,\deg(J').$

The proof of (iii) is given 
in \cite[Theorem~3.5]{paths}.
\end{proof}

The following proposition gives the degree of $V_{\Delta}$ in terms of the degree of $V_{\emptyset}$.

\begin{proposition}\label{prop:degree of Vdelta}
Let~$\mathbbm{1}_{a\ge b}$ be the indicator function of 
$a\ge b$, i.e.,
$\mathbbm{1}_{a\ge b}=1$ if $a\ge b$ and $0$ otherwise.~Let 
\begin{align*}
    &\alpha \coloneqq \sum_{u=t-1}^{\ell-t+1}\binom{\ell}{u}\det\!\bigl[\tbinom{d + u - i - j}{\,d - i\,}\bigr]_{_{i,j \in [t-1]}} \cdot \det\!\bigl[\tbinom{d + \ell-u - i - j}{\,d - i\,}\bigr]_{_{i,j \in [t-1]}}, \;\; 
    \beta \coloneqq \binom{\ell}{t-1}\binom{t-1}{2t-2-\ell} d^{2t-2-\ell}. 
\end{align*}
Then
\begin{itemize}
    \item [{\rm (i)}] if $\ell > 2t-2$, the degree of $V_\Delta$ is
        $\alpha\cdot \mathbbm{1}_{(t-1)(d-t+1) \geq \ell} + \deg(V_\emptyset)\cdot \mathbbm{1}_{(t-1)(d-t+1) \leq \ell},$
    \item [{\rm (ii)}] and if $\ell \leq  2t-2$, the degree of $V_\Delta$ is
       $\beta \cdot \mathbbm{1}_{(d-t)(\ell-t)+d-1 \geq \ell} + \deg(V_\emptyset) \cdot\mathbbm{1}_{(d-t)(\ell-t)+d-1 \leq \ell}.$
\end{itemize} 
\end{proposition}

\begin{proof}
We divide the proof into several cases. In all cases, we make use of Proposition~\ref{prop:useful results}~(i).

\smallskip
\noindent
\textbf{Case~1.} Suppose that $\ell > 2t - 2$ and $(t - 1)(d - t + 1) < \ell$, or that $\ell < 2t-2$ and $(d-t)(\ell-t) +d-1 -\ell<0$.
Then by~\eqref{eq:difference}, $V_{\emptyset}$ is the unique top-dimensional component of $V_{\Delta}$, hence 
$\deg(V_{\Delta}) = \deg(V_{\emptyset}).$

\smallskip
\noindent
\textbf{Case~2.} Suppose that $\ell > 2t - 2$ and $(t - 1)(d - t + 1) \ge \ell$.  
By~\eqref{eq:difference} and Lemma~\ref{top dimensional}, the top-dimensional components of $V_{\Delta}$ are the varieties $V_{S}^{0}$ corresponding to admissible subsets $S$ such that $u + v = \ell$.  
For such $S$, the defining ideal is 
\[
J_{S,0} = \langle x_{i,j}: j\in S, i\in [d] \rangle+ I_{t}(X_{A'}) + I_{t}(X_{B'}),
\]
where 
\[
A' = \bigcup_{j \in A} \{2j - 1, 2j\}, 
\qquad 
B' = \bigcup_{j \in B} \{2j - 1, 2j\}.
\]
The matrices $X_{A'}$ and $X_{B'}$ have sizes $d\times u$ and $d\times(\ell-u)$, respectively. 
By Proposition~\ref{prop:useful results}~(iii),
\[
\deg(I_{t}(X_{A'})) = 
\det\!\bigl[\tbinom{d + u - i - j}{d - i}\bigr]_{i,j = 1}^{t - 1},
\qquad
\deg(I_{t}(X_{B'})) = 
\det\!\bigl[\tbinom{d + \ell - u - i - j}{d - i}\bigr]_{i,j = 1}^{t - 1}.
\]
Since $A \amalg B = [\ell]$, the ideals $I_{t}(X_{A'})$ and $I_{t}(X_{B'})$ involve disjoint sets of variables.  
Hence, by Proposition~\ref{prop:useful results}~(ii),
\[
\deg(V_{S}^{0}) =
\det\!\bigl[\tbinom{d + u - i - j}{d - i}\bigr]_{i,j = 1}^{t - 1}
\cdot
\det\!\bigl[\tbinom{d + \ell - u - i - j}{d - i}\bigr]_{i,j = 1}^{t - 1}.
\]

The admissibility condition requires $u, v \le \ell - t + 1$, which, using $u + v = \ell$, is equivalent to $t - 1 \le u \le \ell - t + 1$.  
For each such pair $(u, v)$, there are precisely $\textstyle \binom{\ell}{u}$ admissible subsets $S$ satisfying $(\ell - |A|, \ell - |B|) = (u, v)$.  
The degree of $V_{\Delta}$ is therefore obtained by summing the degrees of all its top-dimensional components, yielding the expression for~$\alpha$.

\smallskip
\noindent
\textbf{Case~3.} Suppose that $\ell \geq  2t - 2$ and $(t-1)(d-t+1)=\ell$.  
In this case, the top-dimensional components of $V_{\Delta}$ are those described in Case~2, together with $V_{\emptyset}$ itself.  
Hence,
$\deg(V_{\Delta}) = \alpha + \deg(V_{\emptyset}).$

\smallskip
\noindent
\textbf{Case~5.} Assume that $\ell < 2t - 2$ and $(d - t)(\ell - t) + d - 1 - \ell > 0$.  
By~\eqref{eq:difference} and Lemma~\ref{top dimensional}, the top-dimensional components of $V_{\Delta}$ are the varieties $V_{S}^{\,2t - 2 - \ell}$ corresponding to admissible subsets $S$ with $u = v = \ell - t + 1$.  
For each such $S$, the defining ideal $J_{S,\,2t - 2 - \ell}$ from~\eqref{ideal JSj} is given by
\begin{equation}\label{2t-2-l}
J_{S_0, j_0}
= \big\langle x_{i,j} : j \in S,\, i \in [d] \big\rangle 
+ \sum_{u \in A \cap B} I_{2}(X_{C_u}).
\end{equation}
Since all the ideals in~\eqref{2t-2-l} involve disjoint sets of variables, Proposition~\ref{prop:useful results}~(ii) implies that 
\[
\deg(J_{S,\,2t - 2 - \ell}) 
= \prod_{u \in A \cap B} \deg\big(I_{2}(X_{C_u})\big).
\]
Each $X_{C_u}$ is a $d \times 2$ matrix of variables, so by Proposition~\ref{prop:useful results}~(iii), one has $\deg(I_{2}(X_{C_u})) = d$.  
Moreover, $\lvert A \cap B \rvert = 2t - 2 - \ell$, and there are exactly 
$\binom{\ell}{t - 1} \binom{t - 1}{\,2t - 2 - \ell\,}$
admissible subsets $S$ satisfying $u = v = \ell - t + 1$.  
Hence $\deg V_\Delta$ is obtained by summing the degrees of its top-dimensional components, giving~$\beta$.

\smallskip
\noindent
\textbf{Case~6.} Suppose that $\ell < 2t-2$ and $(d-t)(\ell-t) +d-1 -\ell=0$.  
In this case, the top-dimensional components of $V_{\Delta}$ are those described in Case~5, together with $V_{\emptyset}$ itself.  
Hence,
$\deg(V_{\Delta}) = \beta + \deg(V_{\emptyset}).
$
\end{proof}

We next turn to a combinatorial approach for computing $\deg(V_\emptyset)$.
For $d=t$, this yields an explicit formula, while in general the problem is formulated in terms of counting non-intersecting lattice paths.

\subsection{A combinatorial approach to compute the degree of \texorpdfstring{$V_{\emptyset}$}{Vempty}}\label{sec: combinatorial approach to compute degree}
Here, we present a combinatorial approach to computing $\deg V_{\emptyset}$.
The argument proceeds in three steps:
\begin{itemize}
  \item[{\rm (i)}] Reformulate the problem as counting minimal transversals of a certain hypergraph on $[d]\times[2\ell]$;
  \item[{\rm (ii)}] Reduce it to counting minimal transversals of a smaller hypergraph on $[d]\times[\ell]$;
  \item[{\rm (iii)}] Interpret the result in terms of counting \emph{non-intersecting lattice paths} in the general case.
\end{itemize}
We begin by fixing notation:
\[
X = (x_{ij})_{d \times 2\ell}, \qquad
Y = (y_{ij})_{d \times \ell}.
\]

Throughout this section, we identify $X$ with the set $[d] \times [2\ell]$ and $Y$ with $[d] \times [\ell]$.  
We now define the hypergraphs $\mathcal{A}_{t}$, $\mathcal{B}_{t}$ and $\mathcal{C}_{t}$ on these sets, respectively.

\begin{itemize}
    \item $\mathcal{A}_{t} \subset 2^{[d] \times [\ell]}$ is the collection of diagonals of the $t \times t$ submatrices of $Y$;
    \item $\mathcal{B}_{t} \subset 2^{[d] \times [2\ell]}$ is the collection of diagonals of the $t \times t$ submatrices of $X$ whose columns contain at most one index from each pair $\{2i - 1, 2i\}$ for all $i \in [\ell]$.
    \item $\mathcal{C}_{t} \subset 2^{[d] \times [2\ell]}$ is the collection of diagonals of the $2 \times 2$ submatrices of $X$ whose columns are indexed by pairs of the form $\{2j - 1, 2j\}$ for some $j \in [\ell]$.
\end{itemize}

More precisely,
\begin{equation}\label{hypergraph At}
\mathcal{A}_{t} = 
\Bigl\{ 
\{x_{i_1, j_1}, \ldots, x_{i_t, j_t}\} \;\Bigm|\;
1 \le i_1 < \cdots < i_t \le d,\;
1 \le j_1 < \cdots < j_t \le \ell
\Bigr\},
\end{equation}

\begin{equation}\label{hypergraph Bt}
\mathcal{B}_{t} = 
\Bigl\{ 
\{x_{i_1, j_1}, \ldots, x_{i_t, j_t}\} \;\Bigm|\;
\begin{array}{l}
1 \le i_1 < \cdots < i_t \le d,\;
1 \le j_1 < \cdots < j_t \le 2\ell, \\[4pt]
|\{j_1, \ldots, j_t\} \cap \{2i - 1, 2i\}| \le 1 
\ \text{for all } i
\end{array}
\Bigr\},
\end{equation}
\begin{equation}\label{hypergraph Ct}
\mathcal{C}_{t} = 
\Bigl\{   \{\{ x_{i,2j-1}, x_{i',2j}\} : 1 \leq i <i'\leq d, j \in [\ell] \}
\Bigr\},
\end{equation}
and for brevity, set
\begin{equation}\label{hyper B}
\mathcal{B} \coloneqq \mathcal{B}_t \cup \mathcal{C}_t. 
\end{equation}

\medskip  
{\large \bf Step~1.} 
Let $J_{\emptyset}$ denote the ideal of $V_{\emptyset}$. Recall that  
$J_{\emptyset} = I_t(X) + \sum_{j=1}^{\ell} I_2(X_{C_j}).$

Our first goal is to reformulate the problem of computing the degree of $J_{\emptyset}$ in purely combinatorial terms. Specifically, we will reduce it to counting the minimal transversals of a certain hypergraph on the set $[d] \times [2\ell]$. To this end, we begin with the following proposition, which expresses the degree of a squarefree monomial ideal in terms of minimal transversals of an associated hypergraph. 

For convenience, given a subset $\sigma \subset [n]$, we write  
$x^{\sigma} = \prod_{i \in \sigma} x_i.
$

Recall that a subset $A \subset [n]$ is called a \emph{transversal} of a collection of subsets (or hypergraph) $\Delta \subset 2^{[n]}$ if $A \cap e \neq \emptyset$ for every $e \in \Delta$.  
A transversal is \emph{minimal} if it contains no transversal of strictly smaller size.  
We denote by $\mathcal{T}(\Delta)$ the collection of all transversals of $\Delta$, and by $\min(\Delta)$ the set of minimal ones.

\begin{proposition}\label{prop:dimension and degree for squarefree}
Let $I \subset \CC[x_1,\dots,x_n]$ be a squarefree monomial ideal minimally generated by $\{x^{\sigma_1}, \dots, x^{\sigma_k}\}$, where $\{\sigma_1, \dots, \sigma_k\} \subset 2^{[n]}$. Then:
\begin{itemize}
\item[{\rm (i)}] There is a one-to-one correspondence between the minimal primes of $I$ and the minimal (with respect to inclusion) transversals of the hypergraph $\{\sigma_1,\dots,\sigma_k\}$.  
The minimal prime corresponding to a minimal transversal $A$ is  
\[
\mathfrak{m}^A = (x_i : i \in A).
\]
\item[{\rm (ii)}] The Krull dimension of $R/I$ equals  
\[
\min\{ |A| : A \text{ is a minimal transversal of } \{\sigma_1,\dots,\sigma_k\} \}.
\]
\item[{\rm (iii)}] The degree of $I$ equals  
\[
|\{ A \subset [n] : A \text{ is a transversal of } \{\sigma_1,\dots,\sigma_k\} \text{ of minimum size} \}|.
\]
\item[{\rm (iv)}] If $I$ is Cohen--Macaulay, all minimal transversals have the same size, and hence  
\[
\deg(I) = |\{ A \subset [n] : A \text{ is a minimal transversal of } \{\sigma_1,\dots,\sigma_k\} \}|.
\]
\end{itemize}
\end{proposition}

\begin{proof}
We prove only \textup{(i)}, since the remaining statements follow immediately from it.  
Associated to the monomial ideal $I$ is its Stanley--Reisner complex
$\Delta \;=\; \{\tau \subset [n] : x^{\tau} \notin I\},$
see \textup{\cite[Chapter~1]{miller2005combinatorial}}. By \textup{\cite[Theorem~1.7]{miller2005combinatorial}}, the minimal primes of $I$ correspond bijectively to the facets of $\Delta$.  
The minimal prime associated to a facet $\tau\in\Delta$ is
$\mathfrak m^{\tau} = (x_i : i\in\tau^{c}),$
where $\tau^{c}$ denotes the complement of $\tau$.
Thus, it suffices to show that the minimal transversals of $\{\sigma_{1},\dots,\sigma_{k}\}$ are precisely the sets
$\{\tau^{c} : \tau \text{ is a facet of }\Delta\}.
$

We first verify that the transversals of $\{\sigma_{1},\dots,\sigma_{k}\}$ are exactly the sets
$\{\tau^{c} : \tau \in \Delta\}.$
A subset $A\subset[n]$ is a transversal of $\{\sigma_{1},\dots,\sigma_{k}\}$ if and only if  
$A\cap \sigma_{i}\neq\emptyset$ for all $i$, equivalently  
$\sigma_{i}\not\subseteq A^{c}$ for all $i$, or   
$x^{\sigma_{i}} \nmid x^{A^{c}}$ for all $i$.  
Since $I=(x^{\sigma_{1}},\dots,x^{\sigma_{k}})$, this holds exactly when $x^{A^{c}}\notin I$, which by definition means $A^{c}\in\Delta$.
Hence $A$ is a transversal of $\{\sigma_{1},\dots,\sigma_{k}\}$ if and only if $A^{c}$ is a face of $\Delta$.  
The minimal transversals correspond exactly to complements of facets of $\Delta$, completing the proof.
\end{proof}

We now apply this result to our setting.

\begin{proposition}\label{prop:degree-paths}
    The degree of $V_{\emptyset}$ equals the number of minimal transversals of $\mathcal{B}$.  
\end{proposition}

\begin{proof}
    The ideal $J_{\emptyset}$ is generated by all $2$-minors of $X_{\{2i - 1, 2i\}}$, the submatrix of $X$ on the columns $2i-1, i$, for $i \in [\ell]$ together with all $t$-minors of $X$.  
Consider the diagonal term order
\[
x_{11} \prec x_{21} \prec \cdots \prec x_{d1} \prec x_{12} \prec x_{22} \prec \cdots \prec x_{d,2\ell}.
\]
By the results of \cite{seccia2022knutson}, a Gröbner basis for $J_{\emptyset}$ with respect to this order is given by the monomials corresponding to the diagonals, namely
\begin{align}\label{generators initial}
\langle x_{i,2j-1}x_{i',2j} \,:\, &1 \leq i < i' \leq d, \, j \in [\ell]\rangle \;\notag\\
&+\;
\langle x_{i_1,j_1} \cdots x_{i_t, j_t} : 1 \leq i_1 < \cdots < i_t \leq d,\, 1 \leq j_1 < \cdots < j_t \leq 2\ell \rangle.
\end{align}
Hence, the initial ideal $\operatorname{in}_\prec J_{\emptyset}$ is generated by these monomials.  
Although $J_{\emptyset}$ itself is not squarefree, the degree is preserved under taking initial ideals:
\begin{equation}\label{equal degrees}
\deg(J_{\emptyset}) = \deg(\operatorname{in}_\prec J_{\emptyset}).
\end{equation}
Since $\operatorname{in}_\prec J_{\emptyset}$ is squarefree, Proposition~\ref{prop:dimension and degree for squarefree} applies.  
From the description \eqref{generators initial}, the associated hypergraph~is  
$\mathcal{B}=\mathcal{B}_t \cup \mathcal{C}_t$
on $[d] \times [2\ell]$, where $\mathcal{B},\mathcal{B}_t$ and $\mathcal{C}_t$ are defined in~\eqref{hyper B}, \eqref{hypergraph Bt} and \eqref{hypergraph Ct}.    
Then, by Proposition~\ref{prop:dimension and degree for squarefree}(iii) and \eqref{equal degrees},  
\[
\deg(J_{\emptyset}) = \text{number of transversals of minimum size of } \mathcal{B}.
\]
Moreover, since $J_{\emptyset}$ (and hence $\operatorname{in}_\prec J_{\emptyset}$) is Cohen--Macaulay, all minimal transversals have the same size by Proposition~\ref{prop:dimension and degree for squarefree}(iv).  
Therefore, the degree of $J_{\emptyset}$ equals the number of minimal transversals of $\mathcal{B}$.  
\end{proof}

\medskip
{\large\bf Step~2.}
The second step reduces the computation of $\deg V_{\emptyset}$ to counting certain minimal transversals of the hypergraph $\mathcal A_t$ on $[d]\times[\ell]$ defined in~\eqref{hypergraph At}.
We explain how minimal transversals of $\mathcal A_t$, viewed in the matrix $Y$, recover the minimal transversals of $\mathcal B$ in $X$.
For this purpose, we introduce a map that associates to each subset of entries of $X$ a corresponding subset of entries of $Y$.

\begin{definition}
Define the map $\pi : 2^{[d]\times[2l]} \to 2^{[d]\times[\ell]}$ by
\[
  \pi(A) = \{(i,j) \in [d]\times[\ell] : (i,2j-1),(i,2j) \in A\}.
\]
In words, $\pi$ assigns to each subset $A \subset [d]\times[2l]$ the set of indices $(i,j)$ for which both entries $(i,2j-1)$ and $(i,2j)$ belong to $A$.
\end{definition}

We now describe how the map $\pi$ connects the transversals of $\mathcal{B}_{t}$ with those of $\mathcal{A}_{t}$.

\begin{lemma}\label{lem: one from the other}
A subset $A \subset [d]\times[2\ell]$ is a transversal of $\mathcal{B}_{t}$ 
if and only if 
$\pi(A)$ is a transversal~of~$\mathcal{A}_{t}$.
\end{lemma}

\begin{proof}
Suppose first that $\pi(A)$ is a transversal of $\mathcal{A}_{t}$. 
We claim that $A$ is then a transversal of $\mathcal{B}_{t}$. 
Consider any element $\{x_{i_1,j_1}, \ldots, x_{i_t,j_t}\}\in \mathcal{B}_{t}$ 
as in~\eqref{hypergraph Bt}. We need to show that
\begin{equation}\label{nonempty}
A \cap \{x_{i_1,j_1}, \ldots, x_{i_t,j_t}\} \neq \emptyset.
\end{equation}
Since $\pi(A)$ is a transversal of $\mathcal{A}_{t}$, we have
$\pi(A) \cap \{y_{i_1,\lceil j_1/2\rceil}, \ldots, y_{i_t,\lceil j_t/2\rceil}\} \neq \emptyset.$
Thus, for some $k\in[t]$, the element $y_{i_k,\lceil j_k/2\rceil}\in \pi(A)$. 
By definition of $\pi$, this implies
$x_{i_k,\,2\lceil j_k/2\rceil-1}, \, x_{i_k,\,2\lceil j_k/2\rceil} \in A,$
and in particular $x_{i_k,j_k}\in A$, proving~\eqref{nonempty}.
Conversely, suppose that $A$ is a transversal of $\mathcal{B}_{t}$. 
Let $\{y_{i_1,j_1},\ldots,y_{i_t,j_t}\}$ be an arbitrary element of $\mathcal{A}_{t}$. 
We must show that
\begin{equation}\label{nonempty 2}
\pi(A) \cap \{y_{i_1,j_1},\ldots,y_{i_t,j_t}\} \neq \emptyset.
\end{equation}
Assume for contradiction that the intersection is empty, 
so $y_{i_k,j_k}\notin \pi(A)$ for all $k\in[t]$. 
From the definition of $\pi$, this means that for each $k$ at least one of 
$x_{i_k,\,2j_k-1}$ or $x_{i_k,\,2j_k}$ does not lie in $A$. 
Let $r_k \in \{2j_k-1,2j_k\}$ be such that $x_{i_k,r_k}\notin A$. 
Then
\[
1 \leq i_1 < \cdots < i_t \leq d, \quad 
1 \leq r_1 < \cdots < r_t \leq 2\ell, \quad
\size{\{r_1,\ldots,r_t\}\cap\{2i-1,2i\}} \leq 1.
\]
Hence $\{x_{i_1,r_1},\ldots,x_{i_t,r_t}\}\in\mathcal{B}_{t}$. 
But by construction this set is disjoint from $A$, 
contradicting that $A$ is a transversal of $\mathcal{B}_{t}$. 
This contradiction establishes~\eqref{nonempty 2} and completes the proof.
\end{proof}

Recall that we denote by $\mathcal{T}(\Delta)$ the sets of transversals of a collection of subsets $\Delta$.
We now assign to each transversal $A$ of $\mathcal{A}_{t}$ a collection of transversals of $\mathcal{B}$.

\begin{definition}
Let $A$ be a transversal of $\mathcal{A}_{t}$, and set
\begin{equation}\label{equ A'}
A'=\{(i,j)\in [d]\times[2\ell] : (i,\lceil j/2\rceil)\in A\}.
\end{equation}
By construction, $\pi(A')=A$. We define
$\rho(A)=\min\{\,B\in\mathcal{T}(\mathcal{B}) : A'\subset B\,\},$
where $\min$ refers to inclusion-minimal subsets. For each $i$, let $Y_i$ denote the $i$-th column of $Y$, and put
$m_i = \size{A\cap Y_i}.$
We then set
\begin{equation*}\label{multiplicity of A}
m(A) = \prod_{i=1}^{\ell} (d - m_i).
\end{equation*}
\end{definition}

\begin{lemma}\label{lem: multiplicity}
For every transversal $A$ of $\mathcal{A}_{t}$ one has
$\size{\rho(A)}=m(A)=\prod_{i=1}^{\ell}(d-m_{i}).$
\end{lemma}

\begin{proof}
Let $B \subset [d]\times [2\ell]$ be any subset containing $A'$.  
Since $\pi(A')=A$ and $A$ is a transversal of $\mathcal{A}_{t}$,  
Lemma~\ref{lem: one from the other} implies that $B$ is a transversal of $\mathcal{B}_{t}$.  
Hence, for $B$ to be a transversal of $\mathcal{B}$, it remains to transverse the collection of $2$-minors.  
Observe that in each block of columns $\{X_{2i-1},X_{2i}\}$ of $X$, exactly $m_i$ rows are already covered by $A'$.  
Thus, there remain $(d-m_i)$ rows in which the corresponding $2\times 2$ minors must be intersected.  
For a fixed $i$, there are precisely $d-m_i$ minimal ways to achieve this.  
Multiplying over all $i$ gives
$\size{\rho(A)}=\prod_{i=1}^{\ell}(d-m_i),$
as claimed.
\end{proof}

The next proposition provides the claimed connection between the transversals 
of $\mathcal{A}_{t}$ and those of $\mathcal{B}$. Recall that $\min(\Delta)$ denotes 
the set of minimal transversals of a hypergraph $\Delta$.

\begin{proposition}\label{prop:rho}
We have the equality
$\min(\mathcal{B})=\bigcup_{A\in \min(\mathcal{A}_{t})}\rho(A).$
\end{proposition}

\begin{proof}
We begin with the inclusion ``$\supseteq$''.  
Let $B \in \rho(A)$ for some $A \in \min(\mathcal{A}_{t})$, and assume for contradiction that $B$ is not a minimal transversal of $\mathcal{B}$. Then there exists a proper subset $\widetilde{B} \subsetneq B$ which is a transversal of $\mathcal{B}$. By construction, we have
$\pi(\widetilde{B}) \subseteq \pi(B) = \pi(A') = A,$
where $A'$ is defined in~\eqref{equ A'}.  
Since $\widetilde{B}$ is a transversal of $\mathcal{B}$, Lemma~\ref{lem: one from the other} implies that $\pi(\widetilde{B})$ is a transversal of $\mathcal{A}_{t}$.  
Because $\pi(\widetilde{B}) \subseteq A$ and $A \in \min(\mathcal{A}_{t})$, it follows that $\pi(\widetilde{B}) = A$, hence $A' \subseteq \widetilde{B}$. But by definition, $B$ is inclusion-minimal among the transversals of $\mathcal{B}$ containing $A'$, which contradicts the existence of~$\widetilde{B}$.  

We now turn to the reverse inclusion.  
Let $B \in \min(\mathcal{B})$.  
By Lemma~\ref{lem: one from the other}, we know that $\pi(B)$ is a transversal of $\mathcal{A}_{t}$.  
Hence, there exists a minimal transversal $A \in \min(\mathcal{A}_{t})$ such that $A \subseteq \pi(B)$.  
This implies that $A' \subseteq B$, where $A'$ is defined as in~\eqref{equ A'}.  
Since $B$ is inclusion-minimal among the transversals of $\mathcal{B}$ containing $A'$, it follows that $B \in \rho(A)$.  
This establishes the desired inclusion.
\end{proof}

This immediately leads to the desired counting. Recall that $Y_i$ denotes the $i$-th column of $Y$ for all $i\in [\ell]$.

\begin{theorem}\label{thm:sum with multiplicities}
The degree of $V_\emptyset$ is equal to \[\sum_{A\in \min(\mathcal{A}_{t})}\prod_{i=1}^{\ell} (d - \size{A\cap Y_{i}}).\]
\end{theorem}

\begin{proof}
By Proposition~\ref{prop:degree-paths}, it suffices to show that $\size{\min(\mathcal{B})}=\sum_{A\in \min(\mathcal{A}_{t})}m(A).$
By Proposition~\ref{prop:rho} and Lemma~\ref{lem: multiplicity}, 
each $\rho(A)$ contributes exactly $m(A)$ elements.  
Moreover, the families $\rho(A)$ are pairwise disjoint for distinct 
$A \in \min(\mathcal{A}_{t})$.  
Summing over all such $A$ gives the stated formula.
\end{proof}

{\large \bf Step~3.}  
The third and last step consists of reinterpreting the formula for the degree of $V_{\emptyset}$ in terms of counting certain {\em non-intersecting lattice paths} on $Y$, as defined in~\cite{paths}.

Following the notation in \cite{paths}, 
let $a_i = i$ and $b_j = j$ for all $i, j\leq t-1$. This translates to considering $t$-minors of a $d \times \ell$ matrix $Y$. 

\begin{definition} \label{def: WS paths}
    A West-South (WS) \emph{path} on the matrix $Y$ is a sequence of variables $(y_{i_1,j_1}, \ldots, y_{i_m,j_m})$ such that for all $s\in \{2,\ldots,m\}$, we have  $(i_s - i_{s-1}, j_s - j_{s-1}) \in \{(1,0), (0,-1)\}$. A \emph{family of non-intersecting paths} on $Y$ is a collection of $t-1$ WS paths, with the $i$th path starting at $y_{i, \ell}$ 
    and ending at $y_{d, i}$, 
    such that no two paths share an entry.

    If $(i_s - i_{s-1}, j_s - j_{s-1}) = (1,0)$, we call $(y_{{i_{s-1}}}, y_{j_{s-1}})$ a \textit{South location}. Any WS path $P$ is determined by the set of its South locations.
\end{definition}

\begin{lemma}\label{lem:staircase-pattern}
    Fix a minimal transversal $\mathcal{T}$ of $\mathcal{A}_t$. Consider a family of WS paths $P_1, \ldots P_{t-1}$ in $Y \setminus \mathcal{T}$ such that $(\cup_iP_i )\cup \mathcal{T} = Y$.
    Let $Y_{P_i}$ denote the subgrid obtained from $Y$ by removing the entries on, to the left, and above the elements of the path $P_i$. Then for all points $y_{(i_s, j_s)}\in Y_{P_i}$,
    \begin{enumerate}
        \item [{\rm (i)}] if for all points $y_{(i_s,j_q)} \in Y_P$ with $j_q < j_s$, we have $y_{(i_s,j_q)} \in \mathcal{T}$, and $y_{(i_s+1,j_s)} \in \mathcal{T}$, then $y_{(i_s,j_s)} \in \mathcal{T}$;
        \item [{\rm (ii)}] if for all points $y_{(i_q,j_s)} \in Y_P$ with $i_q < i_s$, we have $y_{(i_q,j_s)} \in \mathcal{T}$, and $y_{(i_s,j_s+1)} \in \mathcal{T}$, then $y_{(i_s,j_s)} \in \mathcal{T}$.
    \end{enumerate}
\end{lemma}

\begin{proof}
    To prove (i), assume $a \coloneqq y_{(i_s,j_s)} \not\in \mathcal{T}$ towards a contradiction. Then fix any diagonal $\mathcal{D}$ in $Y$ containing $b\coloneqq y_{(i_s+1,j_s)}$. We will show that $\mathcal{D} \cap \mathcal{T}$ contains at least one element other than $b$, which would mean that $\mathcal{T} \setminus \{b\}$ is also a transversal, contradicting with the minimality of $\mathcal{T}$. Towards another contradiction, assume $\mathcal{D} \cap \mathcal{T} = \{ b\}$. Then $\mathcal{D}$ cannot contain any of the elements $y_{(i_s,j_q)}\in Y_P$ with $j_q < j_s$. Therefore, for all points $y_{(i_q,j_q)} \in \mathcal{D}$ with $j_q < j_s$, i.e. for all points on the diagonal $\mathcal{D}$  to the left of $b$, there is a path $P^{{(i_q,j_q)}}$ containing $y_{(i_q,j_q)}$. Note that a single WS path cannot contain more than one element of the diagonal. Now we obtain a new diagonal $\mathcal{D}'$ from $\mathcal{D}$ as follows:

    If $\mathcal{D}$ contains no element from row $i_s$, then we obtain $\mathcal{D}'$ by replacing $b$ with $a$ in $\mathcal{D}$. Now since $\mathcal{T}$ is a transversal,  $\mathcal{D}' \cap \mathcal{T} \neq \emptyset$. But by the choice of $\mathcal{D}'$, only the elements $y_{(i_q,j_q)} \in \mathcal{D}'$ with $j_q > j_s$ can belong to $\mathcal{T}$. Any of such elements also belongs to $\mathcal{D}$, which is a contradiction to our assumption that $\mathcal{D} \cap \mathcal{T} = \{b\}$.

   For notational simplicity, write $\mathcal{D} = \{d_1, \ldots, d_u = b, \ldots, d_{t}\}$, ordered by increasing row/column index. Suppose  $d_{u-1} = y_{(i_s,j_q)}$ with $j_q < j_s$. Since $P^{(i_s,j_q)}$ is a WS path, we can shift $d_{u-1}$ along $P^{(i_s,j_q)}$, first East and then North, to arrive at the South location immediately before $d_{u-1}$ on the path $P^{(i_s,j_q)}$. Call this South location $d_{u-1}'$. If $d_{u-1}'$ lies in the same row as $d_{u-2}$, we perform the same shift on the path containing $d_{u-2}$ to obtain $d_{u-2}'$; otherwise, set $d_{u-2}'\coloneqq d_{u-2}$.
Iterating this procedure yields $d_1'$.
This is possible since the paths are non-intersecting and each path contains at most one diagonal element.
Define $\mathcal{D}' \coloneqq \{d_1', \ldots, d_{u-1}', a, d_{u+1}, \ldots, d_{t}\}$. Arguing as before, this contradicts the assumption
$\mathcal{D} \cap \mathcal{T} = \{b\}$.
   
    The proof of (ii) is similar to (i).
\end{proof}

\begin{theorem}\label{thm:non-intersection-path}
     A family of $t-1$ WS paths $P_1,\ldots,P_{t-1}$ on $Y$ is non-intersecting if and only if the set $Y \setminus \cup_{i=1}^{t-1} P_i$ is a minimal transversal of $\mathcal{A}_t$.
\end{theorem}
\begin{proof}
    Consider a family of of $t-1$ WS non-intersecting paths $P_1,\ldots,P_{t-1}$ on~$Y$. Toward a contradiction, assume that the complement $Y \setminus \cup_{i=1}^{t-1} P_i$ is not a transversal. Then there exists some $t$-minor of $Y$ with diagonal $\{y_{i_1,j_1},\ldots, y_{i_{t},j_{t}}\}$ such that for every $s \in [t]$, there is a path containing $y_{i_s,j_s}$. However, since all paths are West-South, each of them can contain at most one diagonal point for any diagonal. But we only have $t-1$ paths and $t$ diagonal points, a contradiction. Therefore, $Y \setminus \cup_{i=1}^{t-1} P_i$ is a transversal. 
    
    To show that it is also minimal, we note that $\langle \mathcal{A}_t\rangle$ is Cohen-Macaulay~\cite[Theorem 3.1]{arbarello2013geometry}, and therefore, all the components of the minimal prime decomposition of this ideal have the same number of minimal generators. But these minimal generators are in bijection with the minimal transversals of $\mathcal{A}_t$. So, the minimal transversals have to have the same size.  Now it is straightforward to show that all families of $t-1$ non-intersecting WS paths on $Y$ pass through exactly $d\ell - (d-t+1) (\ell-t+1)$ points of the matrix $Y$. Therefore, their complements will have the same size $(d-t+1)(\ell-t+1)$, and it suffices to show that the complement of one special family of $t-1$ non-intersecting paths forms a minimal transversal. 

    Consider the special collection of paths $P_1,\ldots, P_{t-1}$, where $P_i = \{y_{i,\ell}, y_{i,\ell-1},\ldots, y_{i,i}, y_{i+1,i},\ldots, y_{d, i}\}$ for all $i\in[t-1]$. Then the complement of $Y \setminus \cup_{i=1}^{t-1} P_i$ is a lower-right $(d - t+1)\times (\ell - t+1)$ submatrix $Y'$ of $Y$, which we have now shown is a transversal. Removing 
    any entry $Y'_{t + i-1, t + j-1}$ leaves the leading term of the minor $\det Y_{\{1\ldots,t-1, t+i-1\}, \{1\ldots,t-1, t+j-1\}}$ uncovered, hence  
    $Y'$ is minimal.

Let $\mathcal{T}\subseteq Y$ be a minimal transversal of $\mathcal{A}_t$. Let $P_0$ be the empty path. Assume we have constructed the WS paths $P_1, \ldots, P_{i-1}$. We now construct the WS path $P_i$ as follows:

Let $Y_{P_{i-1}}$ be the subgrid of $Y$ obtained by removing all entries that lie strictly to the left or strictly above any element of $P_{i-1}$. Then $P_i$ is defined recursively within $Y_{P_{i-1}}$ as follows:
starting at $P_i$, for each current position $y_{(a,b)}$ that is not on the last row of $Y_{P_{i-1}}$, let
\[
b^* = \max\{b' \mid b' \le b,\; y_{(a,b')} \in \partial Y_{P_{i-1}} \cup \mathcal{T}\},
\]
and set the next South location (vertex of the path $P_{i}$) to be $y_{(a,b^*+1)}$. By Lemma \ref{lem:staircase-pattern}, the entries of the transversal in $Y_{P_{i-1}}$ form a staircase pattern on the upper-left corner of the grid. Therefore, the resulting sequence of entries forms a valid WS path $P_i$ that terminates at $Q_i$.
\end{proof}

\begin{remark}
Theorem~\ref{thm:non-intersection-path} and Proposition~\ref{prop:dimension and degree for squarefree}(i)
yield an alternative proof of~\cite[Theorem~3.3]{paths}.
\end{remark}

Combining Proposition~\ref{thm:sum with multiplicities} and Theorem~\ref{thm:non-intersection-path}, we obtain the following.
\begin{proposition}\label{prop:counting lattice paths}
Let $Y_i$ denote the $i$th column of $Y$.
Then
\[
\deg V_{\emptyset}
=
\sum_{\substack{(P_1,\ldots,P_{t-1})\\ \text{non-intersecting WS 
}}}
\;
\prod_{i=1}^{\ell}
\bigl|(\bigcup_{j=1}^{t-1} P_j)\cap Y_i\bigr|.
\]
\end{proposition}

\subsection{Lattice paths and the degree formula}\label{sub:lattice_path}

In the previous subsections, we have reduced the problem of computing the degree of $V_\Delta$ to counting lattice paths with multiplicities, as described in Proposition~\ref{prop:counting lattice paths}. Here, we count these paths and provide an explicit formula for the degree of $V_\Delta$.

\begin{proposition}
Let the generating function $\mathcal{G}_{d,\ell,t}(z_1,\dotsc,z_\ell)$
for $t-1$ lattice paths in the $d \times \ell$ matrix $Y = (y_{ij})$,
be defined by 
\begin{equation}
\mathcal{G}_{d,\ell,t}(z_1,\dotsc,z_\ell) \coloneqq 
\sum_{\substack{\mathbf{P}=(P_1,\dotsc,P_{t-1})\\\text{non-intersecting}}} \prod_{j=1}^{\ell} z_j^{|C_j \cap \mathbf{P}|}.
\end{equation}
Then 
\begin{equation}\label{eq:vertexLGV}
\mathcal{G}_{d,\ell,t}(z_1,\dotsc,z_\ell) =
\left(\prod_{j=1}^\ell z_j^{\min(j,t-1)} \right) \; \cdot \;
\det \left( 
h_{d-i}(z_j,z_{j+1},\dotsc,z_\ell)
\right)_{1\leq i,j \leq t-1},
\end{equation}
where $h_i$ is the complete homogeneous symmetric polynomial of degree $j$.
 
\end{proposition}
\begin{proof}
We use the Lindström--Gessel--Viennot lemma for 
computing the generating function of non-intersecting lattice paths, see \cite{lindstrom1973vector, gessel1985binomial}. 
We consider the $d {\times} \ell$ grid where there 
are $t-1$ starting vertices
on the East border and $t-1$ end vertices 
on the South border, see Figure~\ref{fig:lgvPaths}.
We consider SW-paths with $t-1$ starting vertices, and $t-1$
ending vertices, given by 
\[
\left\{ (j,\ell) \right\}_{j=1}^{t-1} \qquad \text{ and } \qquad
\left\{ (d,i) \right\}_{i=1}^{t-1},
\]
respectively. Note, vertices in the grid 
are indexed using (row, column) convention.

A path from the $i$th starting vertex to the $j$th ending vertex 
consists of $d+1-i$ South steps and $\ell+1-j$ East steps.
If we weight such a path with $z_j$ for each South step $(k,j) \to (k+1,j)$, then such a path is constructed 
by choosing (with repetition) $d-i$
South steps in the columns $\{j,j+1,\dotsc,\ell\}$.
Thus, the total (weighted) sum of all possible paths
is given by the complete homogeneous symmetric function
$h_{d-i}(z_j,z_{j+1},\dotsc,z_\ell)$.
The Lindström--Gessel--Viennot lemma now gives that
the total sum over all non-intersecting $(t-1)$-tuples of paths
is given by the determinant 
\[
\det \left( 
h_{d-i}(z_j,z_{j+1},\dotsc,z_\ell)
\right)_{1\leq i,j \leq t-1.}
\]
However, we weight paths by the occupied \emph{vertices} in each column.
We need two observations: 
\begin{itemize}
\item the first $\min(j,t-1)$ paths intersect column $j$, whenever $1 \leq j \leq n$.
\item the number of vertices covered by a path in column $j$,
is given by the number of paths intersecting this column,
plus the total number of South steps in this column.
\end{itemize}
These observations now implies \eqref{eq:vertexLGV}.
For example, $P_1$, $P_2$ and $P_3$ in Figure~\ref{fig:lgvPaths}
have vertex weights 
\[
P_1:  z_1^4 z_2 z_3^2 z_4  z_5^3 z_6 z_7   z_8 z_9, \quad
P_2:        z_2^3 z_3 z_4 z_5^2 z_6^2  z_7^2 z_8 z_9 , \quad
P_3:              z_3 z_4^2 z_5 z_6    z_7^3 z_8^2 z_9.   \qedhere
\]\end{proof}
\vspace{-2em}
\begin{figure}[!ht]
\centering
\scalebox{.75}{
\def\StackHeight{0.35\textheight}

\begin{minipage}[c]{0.35\linewidth}
  \centering
  \includegraphics[height=0.20\textheight,keepaspectratio,page=2]{lattice-paths}

 \vspace{-0.02\textheight}

  \includegraphics[height=0.20\textheight,keepaspectratio,page=3]{lattice-paths}
\end{minipage}\hfill
\begin{minipage}[c]{0.6\linewidth}
  \centering
  \makebox[\linewidth][c]{%
    \includegraphics[height=\StackHeight,keepaspectratio,page=1]{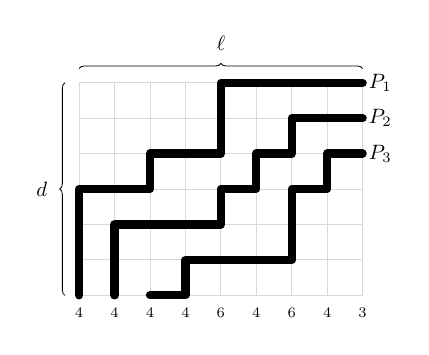}%
  }
\end{minipage}}

\caption{(Left) Two families of WS non-intersecting paths from Remark~\ref{rem:two_paths}. (Right) Three lattice paths in the case $\ell=9$, $d=7$ and $t-1=3$. 
    The path $P_1$ has edge weight $z_1^3 z_3 z_5^2$, $P_2$ has weight $z_2^2 z_5 z_6 z_7$ and finally $P_3$ has weight $z_4 z_7^2 z_8$.}
\label{fig:lgvPaths}

\end{figure}
\begin{proposition}\label{cor:total_transversal}
The 
degree of $V_{\emptyset}$ is equal to
\[
  \left.\frac{\partial}{\partial z_1} \dotsb \frac{\partial}{\partial z_\ell}
  \mathcal{G}_{d,\ell,t}(z_1,\dotsc,z_\ell) \right\vert_{z_1=\dotsb = z_\ell =1.}
\]
\end{proposition}
\begin{proof}
By Proposition~\ref{prop:counting lattice paths}, the degree of $J_{\emptyset}$ is equal to the number of families of $t-1$ WS non-intersecting paths $\{P_1,\ldots, P_{t-1}\}$, counted with multiplicity
$\size{(\bigcup_{j=1}^{t-1}{P_j}) \cap Y_{i}}$.
Taking partial derivatives of $\mathcal{G}_{d,\ell,t}$ produces these multiplicities.
\end{proof}
For example, in Figure~\ref{fig:lgvPaths}, the total weight of the three lattice paths is $z_1^4 z_2^4 z_3^4 z_4^4 z_5^6 z_6^4 z_7^6 z_8^4 z_9^3$.
The partial derivatives then multiplies this monomial with the factor 
$4\cdot 4 \cdot 4 \dotsm 3$, which is exactly the multiplicity~$m(A)$.

\begin{remark}\label{rem:two_paths}
Note that some coefficients might be larger than $1$ in $\mathcal{G}_{d,\ell,t}(z_1,\dotsc,z_\ell)$. 
For example, when $d=4$, $\ell=5$, $t=3$ and $(s_1,\dotsc,s_5)=(3,3,3,3,2)$, 
there are two distinct families of non-intersecting paths, as illustrated in Figure~\ref{fig:lgvPaths} (right).
\end{remark}

\begin{corollary}\label{cor:count}
Suppose $d=t$. Then the degree of $V_{\emptyset}$ is given by
\[
\deg(V_{\emptyset})
\;=\;
d^{\,d-1}(d-1)^{\,\ell-d+1}\binom{\ell}{d-1}.
\]
\end{corollary}

\begin{proof}
We first reduce the statement to a purely combinatorial enumeration.
By Proposition~\ref{prop:counting lattice paths} and Corollary~\ref{cor:total_transversal}, 
the degree $\deg(V_{\emptyset})$
is computed by the contribution of the
non-intersecting path families encoded by the generating function $\mathcal{G}_{d,\ell,t}(z_1,\dots,z_\ell)$. 

To make this reduction explicit, fix a monomial $\prod_{j=1}^{\ell} z_j^{e_j}$ of $\mathcal{G}_{d,\ell,t}(z_1,\dots,z_\ell)$
corresponding to a non-intersecting family of $(d-1)$ paths.
For $t=d$, each column contributes at least $d-1$ occupied vertices,
and any additional contribution arises from a south step in that column.
Non-intersection implies that at most one south step may occur in any column.
Hence $e_j\in\{d-1,d\}$ for all $j$, and since the total number of south steps
equals $\ell-d+1$, exactly $\ell-d+1$ indices satisfy $e_j=d-1$.

It follows that every partial derivative $\partial_{z_1}\cdots\partial_{z_\ell}$ 
contributes the same factor 
$d^{\,d-1}(d-1)^{\,\ell-d+1}$. 
Evaluating at $z_1=\cdots=z_\ell=1$, the number of minimal transversals equals
\[
d^{\,d-1}(d-1)^{\,\ell-d+1}
\cdot
\left|
\binom{(d-j)+\ell-i}{d-j}
\right|_{1\le i,j\le d-1}.
\]

Reindexing gives
\[
\left|
\binom{(d-j)+\ell-i}{d-j}
\right|_{1\le i,j\le d-1}
=
\left|
\binom{i+j+1+(\ell-d+1)}{j+1}
\right|_{0\le i,j\le d-2}.
\]
By \cite[Theorem~1.1]{Bacher2002}, this determinant evaluates to $\binom{\ell}{d-1}$, which completes the proof. 
\end{proof}

\section{Connections to 
quasi-products of matroids}\label{sec:quasi_product}

In this final section, we place our results in a broader matroid-theoretic context.
Although this paper focuses on the case $s=2$, the general conditional independence varieties $V_{\Delta^{s,t}}$ admit a natural interpretation in terms of quasi-products of matroids when $d \le s+t-3$.

We briefly recall the general setup in order to make this connection precise.
Let $s \le k$, $t \le \ell$, and $d$ be positive integers, and let $\mathcal{Y}$ denote the $k \times \ell$ matrix of integers defined in~\eqref{matrix}.  
Consider the hypergraph
$\Delta^{s,t}
\;\coloneqq \;
\bigcup_{i \in [k]} \binom{R_i}{t}
\;\cup\;
\bigcup_{j \in [\ell]} \binom{C_j}{s}.$
on the vertex set $[k\ell]$. The associated CI variety is
\[
V_{\Delta^{s,t}}
\; \coloneqq \;
\bigl\{
Y \in \CC^{d \times k\ell}
:
\operatorname{rank}(Y_F) < |F|
\text{ for all } F \in \Delta^{s,t}
\bigr\},
\]
where $Y_F$ denotes the submatrix of $Y$ consisting of the columns indexed by $F$.

The following result from~\cite{clarke2021matroid} makes this connection precise by describing
$V_{\Delta^{s,t}}$ in matroid-theoretic terms.

\begin{theorem}\label{thm:dleq s+t-3}
Let $s \le k$, $t \le \ell$, and $d$ be positive integers with $d \le s+t-3$. Then
\begin{equation}\label{circuits}
\mathcal{C}
\;=\;
\min (
\Delta^{s,t}
\;\cup\;
\textstyle\binom{[k\ell]}{d+1}
)
\end{equation}
defines the collection of circuits of a matroid $M_{k,\ell}^{s,t,d}$ of rank $d$ on $[k\ell]$, where $\min$ denotes the collection of inclusion-minimal subsets.
\end{theorem}
As a consequence, when $d \le s+t-3$ we have
\[
V_{\Delta^{s,t}}
=
\bigl\{
\gamma = (\gamma_1,\ldots,\gamma_{k\ell}) \in (\CC^d)^{k\ell}
:
\mathcal{M}_\gamma \le M_{k,\ell}^{s,t,d}
\bigr\},
\]
where $\le$ denotes the weak partial order on matroids.  
This characterization naturally connects $V_{\Delta^{s,t}}$ to the notion of \emph{quasi-products of matroids}, introduced in~\cite{las1981products}.

\begin{definition}
Let $M$ and $N$ be matroids on $[k]$ and $[\ell]$, respectively. A matroid $Q$ on $[k\ell]$ is called a \emph{quasi-product} of $M$ and $N$ if, for each $i \in [k]$ and $j \in [\ell]$, the restrictions
$Q\mid R_i$
and
$Q\mid C_j$
are isomorphic to $N$ and $M$, respectively, via the natural bijections.  
If, in addition,
$\rank(Q)=\rank(M)\,\rank(N),$
then $Q$ is called a \emph{tensor product} of $M$ and $N$.
\end{definition}

In general, tensor products of matroids need not exist or be unique.  
As a direct consequence of the circuit description of $M_{k,\ell}^{s,t,d}$ in~\eqref{circuits}, we obtain the following corollary.
Here, $U_{d,n}$ denotes the uniform matroid on $[n]$, whose independent sets are all subsets of size at most~$d$.

\begin{corollary}
If $d \le s+t-3$, then the matroid $M_{k,\ell}^{s,t,d}$ is a quasi-product of 
$U_{s-1,k}$ and $U_{t-1,\ell}$.
\end{corollary}

Moreover, $M_{k,\ell}^{s,t,d}$ is the unique maximal matroid among all quasi-products of $U_{s-1,k}$ and $U_{t-1,\ell}$ of rank~$d$.

Tensor products of matroids play an important role in rigidity theory, and it was conjectured in~\cite{mason1981glueing} that, whenever a tensor product of two matroids exists, there is a unique maximal one in the weak order.
From this viewpoint, the decomposition of the CI variety $V_{\Delta^{s,t}}$ 
can be viewed as an algebraic-geometric analogue of these matroid-theoretic questions, a systematic study of which is left for future work.

\noindent\textbf{Acknowledgements.}
We thank Tibor Jordán and Kristóf Bérczi for insightful discussions on quasi-products of matroids, which motivated Section~\ref{sec:quasi_product}. We are grateful to Aldo Conca for pointing us to the references \cite{arbarello2013geometry,paths}.~F.M. gratefully acknowledges the hospitality of the Department of Mathematics at 
Stockholm University and thanks Samuel Lundqvist and the Wenner-Gren Foundation for their support during her research visit.
This work was supported by the FWO grants G0F5921N (Odysseus), G023721N, and 1126125N, and by the KU Leuven grant iBOF/23/064.~P.S.~was supported by the Vanier Canada Graduate Scholarship and the British Columbia Graduate Scholarship.

\bibliographystyle{plain}
\bibliography{Citation}

\medskip
{\small\noindent {\bf Authors' addresses}}
\medskip

\noindent{Per Alexandersson, Stockholm University} \hfill {\tt per.alexandersson@math.su.se}
\\
\noindent{Yulia Alexandr,  UCLA}\hfill {\tt yulia@math.ucla.edu}
\\ 
\noindent{Emiliano Liwski, KU Leuven} \hfill {\tt emiliano.liwski@kuleuven.be}
\\
\noindent{Fatemeh Mohammadi, 
KU Leuven} \hfill {\tt fatemeh.mohammadi@kuleuven.be}
\\ 
\noindent{Pardis Semnani, University of British Columbia}\hfill {\tt  psemnani@math.ubc.ca} 
\\

\end{document}